\documentclass[a4paper,11pt]{amsart}

\usepackage{a4wide}
\usepackage[foot]{amsaddr}
\usepackage{mathabx}

\usepackage[breaklinks,bookmarks=false]{hyperref}
\hypersetup{colorlinks,linkcolor=blue,citecolor=blue,urlcolor=blue,plainpages=false}

\usepackage{amssymb}
\usepackage{mathrsfs}

\newcommand{\eq}{:=}

\newcommand{\grad}{\boldsymbol \nabla}
\renewcommand{\div}{\grad \cdot}
\newcommand{\curl}{\grad \times}

\newcommand{\ccurl}{\boldsymbol{\operatorname{curl}}}
\newcommand{\ddiv}{\operatorname{div}}

\newcommand{\jmp}[1]{[\![#1]\!]}
\newcommand{\avg}[1]{\{\!\!\{#1\}\!\!\}}

\newcommand{\BA}{\boldsymbol A}

\newcommand{\BH}{\boldsymbol H}
\newcommand{\BI}{\boldsymbol I}

\newcommand{\BL}{\boldsymbol L}
\newcommand{\BM}{\boldsymbol M}

\newcommand{\BX}{\boldsymbol X}

\newcommand{\ba}{\boldsymbol a}
\newcommand{\bb}{\boldsymbol b}

\newcommand{\be}{\boldsymbol e}
\newcommand{\bg}{\boldsymbol g}

\newcommand{\bn}{\boldsymbol n}
\newcommand{\bo}{\boldsymbol o}

\newcommand{\br}{\boldsymbol r}

\newcommand{\bv}{\boldsymbol v}
\newcommand{\bw}{\boldsymbol w}
\newcommand{\bx}{\boldsymbol x}

\newcommand{\CE}{\mathcal E}
\newcommand{\CF}{\mathcal F}

\newcommand{\CI}{\mathcal I}
\newcommand{\CJ}{\mathcal J}

\newcommand{\CP}{\mathcal P}

\newcommand{\CT}{\mathcal T}

\newcommand{\CV}{\mathcal V}

\newcommand{\LE}{\mathscr E}

\newcommand{\BCP}{\boldsymbol{\CP}}

\newtheorem{theorem}{Theorem}
\newtheorem{lemma}[theorem]{Lemma}
\newtheorem{remark}[theorem]{Remark}
\newtheorem{corollary}[theorem]{Corollary}

\numberwithin{theorem}{section}
\numberwithin{equation}{section}

\newcommand{\artlongtitle}{A new family of a posteriori error estimates for non-conforming finite element methods leading to stabilization-free error bounds}
\newcommand{\artshorttitle}{Stabilization-free a posteriori error estimates}
\newcommand{\artauthor}{T. Chaumont-Frelet}
\newcommand{\artaffiliation}{Inria, Univ. Lille, CNRS, UMR 8524 -- Laboratoire Paul Painlev\'e}

\newcommand{\LBA}{\boldsymbol \Lambda}
\newcommand{\lba}{\boldsymbol \lambda}

\newcommand{\AAA}{\underline{\BA}}
\newcommand{\bzero}{\bo}

\newcommand{\omF}{\omega_F}
\newcommand{\omK}{\omega_K}
\newcommand{\tomK}{\widetilde \omega_K}
\newcommand{\ttomK}{\widetilde{\widetilde \omega}_K}
\newcommand{\omKF}{\widehat \omega_K}
\newcommand{\omKE}{\check \omega_K}
\newcommand{\gaeD}{\gamma_{\ell,{\rm D}}}

\newcommand{\TCP}{\underline{\BCP}}

\newcommand{\GD}{\Gamma_{\rm D}}
\newcommand{\GN}{\Gamma_{\rm N}}

\newcommand{\CFIh}{\CF^{\rm I}_h}
\newcommand{\CFDh}{\CF^{\rm D}_h}
\newcommand{\CFNh}{\CF^{\rm N}_h}

\newcommand{\CLA}{C_{\rm L}}

\newcommand{\GRAD}{\mathfrak G}
\newcommand{\LIFT}{\mathfrak L}

\newcommand{\bsig}{\boldsymbol \sigma}
\newcommand{\btau}{\boldsymbol \tau}
\newcommand{\bxi}{\boldsymbol \xi}
\newcommand{\bphi}{\boldsymbol \phi}
\newcommand{\bpsi}{\boldsymbol \psi}
\newcommand{\bthe}{\boldsymbol \theta}
\newcommand{\bThe}{\boldsymbol \Theta}
\newcommand{\brho}{\boldsymbol \rho}

\newcommand{\RT}{\boldsymbol{RT}}
\newcommand{\ND}{\boldsymbol{N}}

\newcommand{\he}{h_{\ome}}
\newcommand{\pe}{\boldsymbol{\psi}_\ell}
\newcommand{\ome}{\omega_\ell}
\newcommand{\CTe}{\CT_h^\ell}

\newcommand{\te}{\boldsymbol{\tau}_\ell}

\newcommand{\ha}{h_{\oma}}
\newcommand{\pa}{\psi_{\ba}}
\newcommand{\oma}{\omega_{\ba}}
\newcommand{\toma}{\widetilde{\omega}_{\ba}}
\newcommand{\CTa}{\CT_h^{\ba}}
\newcommand{\CVa}{\CV_h^{\ba}}
\newcommand{\tCTa}{\widetilde{\CT}_h^{\ba}}

\newcommand{\tgrad}{\underline{\grad}}

\newcommand{\cor}{\bthe}

\newcommand{\cank}{\be^k}

\begin{document}

\title[\artshorttitle]{\artlongtitle}
\author{\artauthor$^\star$}

\address{\vspace{-.5cm}}
\address{\noindent \tiny \textup{$^\star$\artaffiliation}}

\begin{abstract}
We propose new a posteriori error estimators for non-conforming finite element discretizations
of second-order elliptic PDE problems. These estimators are based on novel reformulations
of the standard Prager--Synge identity, and enable to prove efficiency estimates without
extra stabilization terms in the error measure for a large class of discretization schemes.
We propose a residual-based estimator for which the efficiency constant scales optimally in
polynomial degree, as well as two equilibrated estimators that are polynomial--degree--robust.
One of these two estimators further leads to asymptotically constant-free error bounds.

\vspace{.25cm}

\end{abstract}

\maketitle
\thispagestyle{empty}


\section{Introduction}

In this work we are interested in finite element and discontinunous Galerkin
discretizations of the following second-order elliptic PDE problem:
\begin{equation}
\label{eq_poisson_strong_intro}
\left \{
\begin{array}{rcll}
-\div(\AAA\grad u) &=& f & \text{ in } \Omega,
\\
u &=& 0 & \text{ on } \GD,
\\
\AAA \grad u \cdot \bn &=& 0 & \text{ on } \GN.
\end{array}
\right .
\end{equation}
This problem is set in a three dimensional domain $\Omega \subset \mathbb R^3$
with its boundary $\partial \Omega$ partitioned into two disjoint subsets
$\GD$ and $\GN$. The right-hand side $f: \Omega \to \mathbb R$
and the diffusion coefficient $\AAA: \Omega \to \mathbb R^{3 \times 3}$
are given, whereas $u: \Omega \to \mathbb R$ is the unknown.
Here, we wish to numerically approximate the gradient $\GRAD \eq \grad u$ of $u$
by a piecewise polynomial ansatz $\GRAD_h$ on a simplicial mesh $\CT_h$ of $\Omega$.
More specifically, we focus on a posteriori error estimation, where the goal
is to provide a fully computable quantity $\eta$ which can be used to certify
the discretization error $\GRAD-\GRAD_h$ in a suitable norm.

An alternative way to state~\eqref{eq_poisson_strong_intro} is that
$\GRAD$ is the unique vector field such that (i) $(-\AAA\GRAD) \cdot \bn = 0$
on $\GN$ and $\div (-\AAA\GRAD) = f$ in $\Omega$, and
(ii) there exists a scalar potential $v$ such that $v = 0$ on $\GD$
and $\GRAD = \grad v$. As a result, a possible starting point for
the design of the estimator $\eta$ is the following Prager--Synge identity
\begin{equation}
\label{eq_prager_synge_intro}
\|\GRAD-\GRAD_h\|_{\AAA,\Omega}^2
=
\min_{\substack{\bsig \in \BH_{\GN}(\ddiv,\Omega) \\ \div \bsig = f}}
\|\AAA^{-1} \bsig+\GRAD_h\|_{\AAA,\Omega}^2
+
\min_{v \in H^1_{\GD}(\Omega)} \|\grad v-\GRAD_h\|_{\AAA,\Omega}^2
\end{equation}
which measures the deviation of the discrete approximation from (i) and (ii).
The notation in~\eqref{eq_prager_synge_intro} are standard, and rigorously
introduced below. The seminal work from which the identity in~\eqref{eq_prager_synge_intro}
originates is~\cite{prager_synge_1947a}. This identity also appears
in~\cite{ern_vohralik_2015a} under an extra assumption on $\GRAD_h$
and in a slightly different form in~\cite{pencheva_vohralik_wheeler_wildey_2013a}.
For completeness, it is rigorously established for general $\BL^2(\Omega)$ vector-fields
below.

For conforming methods, one looks for an approximation $u_h$ of $u$ in a subspace
of $H^1_{\GD}(\Omega)$. As a result, setting $\GRAD_h \eq \grad u_h$, the second
term in~\eqref{eq_prager_synge_intro} vanishes. In this case, residual-based
and flux-equilibrated estimator techniques are very well established and
control the first term in the right-hand side of~\eqref{eq_prager_synge_intro}
provided $\GRAD_h$ satisfies suitable orthogonality conditions, which in turn
is guaranteed by the numerical method producing $u_h$. We refer the reader
to~\cite{braess_pillwein_schoberl_2009a,melenk_wohlmuth_2001a} for an overview.
In contrast, for mixed and non-conforming methods, the second term
in~\eqref{eq_prager_synge_intro} does not vanish and must be controlled.

In cases where $\GRAD_h = \grad_h u_h$ is element-wise given by the gradient
of a (known) discrete non-conforming potential $u_h$,
a standard procedure is to postprocess the potential $u_h$ into a conforming version
$\widetilde{u}_h$, leading to the fully explicit estimation
\begin{equation}
\label{eq_potential_intro}
\min_{v \in H^1_{\GD}(\Omega)} \|\grad v-\GRAD_h\|_{\AAA,\Omega}
\leq
\|\grad \widetilde{u}_h-\GRAD_h\|_{\AAA,\Omega}
\end{equation}
of the second term in the right-hand side of~\eqref{eq_prager_synge_intro}.
This approach is employed, e.g., in~\cite{ainsworth_fu_2018a,alaoui_ern_2004a,ern_vohralik_2015a}.
An important observation is that since $\widetilde{u}_h$ is typically obtained
from averaging the values of $u_h$ across neighboring elements, $\grad \widetilde{u}_h$
is not (in general) solely defined in terms of $\GRAD_h$. Alternatively, if one does not
wish to explicitly compute $\widetilde{u}_h$, one might simply invoke the existence of
such a conforming function, and use the estimate
\begin{equation}
\label{eq_jump_intro}
\|\grad \widetilde{u}_h-\GRAD_h\|_{\AAA,\Omega}^2
\leq
C \sum_{K \in \CT_h} \frac{p^2}{h_K} \|\jmp{u_h}\|_{\partial K \setminus \GN}^2
\end{equation}
which again provides the desired bound in~\eqref{eq_prager_synge_intro}.
Here, $\jmp{u_h}$ denotes the jump of $u_h$, $p$ is the polynomial degree,
and $h_K$ is the diameter of the element $K \in \CT_h$. We note that in this case
too, the estimator cannot (in general) be determined only in terms of $\GRAD_h$.
For simplicity, we will respectively refer to the right-hand side of~\eqref{eq_potential_intro}
and~\eqref{eq_jump_intro} as the standard equilibrated and residual-based estimators
in the remainder of this introduction.

The fact that the estimators in~\eqref{eq_potential_intro} and~\eqref{eq_jump_intro}
are not solely defined in terms of $\GRAD_h$ can lead to complications. First of all,
if $\GRAD_h$ is not given as the gradient of a discrete potential, the above techniques
cannot readily apply. In addition, to the best of the author's knowledge, general error
lower bounds of the form
\begin{equation*}
\frac{p^2}{h_K} \|\jmp{u_h}\|_{\partial K \setminus \GN}
\leq
C \|\grad \widetilde{u}_h-\GRAD_h\|_{\AAA,\omK},
\end{equation*}
where $\omK$ is a local region around $K$, have never been established in the literature.
It means that in general, estimators defined using a potential are not efficient in
the natural error measure.

In this work, we propose different characterizations of the term measuring the
level of non-conformity in~\eqref{eq_prager_synge_intro}. When the domain $\Omega$
and the boundary decomposition correspond to a topologically trivial
setting, the following identities hold true for any vector field $\GRAD_h \in \BL^2(\Omega)$
\begin{equation}
\label{eq_new_identities}
\min_{v \in H^1_{\GD}(\Omega)} \|\grad v-\GRAD_h\|_{\AAA,\Omega}
=
\min_{\bphi \in \BH_{\GD}(\ccurl^0,\Omega)} \|\bphi-\GRAD_h\|_{\AAA,\Omega}
=
\max_{\substack{\bthe \in \BH_{\GN}(\ccurl,\Omega) \\ \|\curl \bthe\|_{\AAA^{-1},\Omega} = 1}}
(\GRAD_h,\curl \bthe)_\Omega.
\end{equation}
Here, we are going to show that for many numerical schemes of interest,
we can still upper bound the left-hand side of~\eqref{eq_new_identities}
by any of the two right-hand sides, up to a constant that only depends on
the shape-regularity of the mesh, and tends to one as the mesh is refined.

We respectively employ the first and second identities in~\eqref{eq_new_identities} to derive
new equilibrated and residual-based error estimators for non-conforming finite element
methods, with special emphasis on Raviart--Thomas mixed discretization and the interior
penalty discontinuous Galerkin (IPDG) scheme. For the equilibrated estimator, we propose
a construction of a curl-free field $\bphi_h$ built from localized N\'ed\'elec finite element
problems to be used in place of the minimizer $\bphi$ in~\eqref{eq_new_identities}, whereas
the residual based estimator takes the form
\begin{equation}
\label{eq_residual_based_intro}
\min_{v \in H^1_{\GD}(\Omega)} \|\grad v-\GRAD_h\|_{\AAA,\Omega}
\leq
C
\sum_{K \in \CT_h} \left \{
\frac{h_K^2}{p^2} \|\curl \GRAD_h\|_K^2
+
\frac{h_K}{p} \|\jmp{\GRAD_h} \times \bn\|_{\partial K \setminus \GN}^2
\right \}.
\end{equation}
We additionally propose an alternative flux reconstruction technique hinging on a further
reformulation of the curl-constrained minimization problem in~\eqref{eq_new_identities}.

In all cases, the proposed estimators are shown to be efficient in the natural energy norm,
i.e., without adding any term linked to the potential to the local error measure
$\|\GRAD-\GRAD_h\|_{\AAA,\omK}$. In addition, the equilibrated error estimators are
polynomial-degree-robust, whereas the efficiency constant of the residual-based estimator
grows with the optimal known rate with the polynomial degree. We finally emphasize
that $\GRAD_h$ does not need to be given element-wise as the gradient of a potential.
Below, we highlight the advantage of the proposed estimators as compared to standard
ones in~\eqref{eq_potential_intro} and~\eqref{eq_jump_intro}. Specifically, we focus
on (a) non-conforming methods of Crouzeix--Raviart type, (b) mixed methods, and (c)
discontinuous Galerkin methods.

(a) We first discuss the case where $\GRAD_h = \grad_h u_h$
and the average of the jump vanishes over each mesh faces, i.e.
$(\jmp{u_h},1)_F = 0$ for all mesh face $F$. This applies, e.g., to
Crouzeix--Raviart (see~\cite{crouzeix_raviart_1973a}) and primal hybrid
(see~\cite{raviart_thomas_1977b})
discretizations of~\eqref{eq_poisson_strong_intro}. In this case, for
the standard residual-based estimator, we can employ a face-wise Poincar\'e
inequality (see e.g.~\cite{bebendorf_2003a}) leading to the estimate
\begin{equation}
\label{eq_poincare_intro}
\sum_{K \in \CT_h} \frac{p^2}{h_K} \|\jmp{u_h}\|_{\partial K \setminus \GN}^2
\leq
C
\sum_{K \in \CT_h} p^2 h_K \|\jmp{\grad u_h} \times \bn_K\|_{\partial K \setminus \GN}^2.
\end{equation}
As a result, the standard residual-based estimator is indeed efficient since the right-hand side
essentially corresponds to the estimator proposed in this work. However, we can see
that the right-hand side is multiplied by $p^3$ as compared to~\eqref{eq_residual_based_intro},
indicating that the estimator may not behave optimally as $p$ increases. As a result,
our residual-based estimator exhibit an improved $p$-dependence over the standard one.
When a potential reconstruction is carefully constructed, the construction
in~\cite{ern_vohralik_2015a,ern_vohralik_2021a} is already polynomial-degree-robust.
Therefore, for equilibrated estimators, although we expect that the present approach
brings new insight, the present work does not lead to quantitative improvements in practice.

(b) For mixed finite element methods, it is natural to consider $\GRAD_h = \AAA^{-1}\bsig_h$,
where $\bsig_h$ is the Raviart--Thomas approximation of the flux. Since
$\GRAD_h$ is not a piecewise gradient, the standard procedure described
above cannot readily apply. In~\cite{cockburn_zhang_2014a} and~\cite{ern_vohralik_2015a},
the standard residual-based and equilibrated estimators are therefore applied to
a postprocessing $\grad_h u_h^\star$ of $\bsig_h$. In both cases, the estimator
can be shown to be efficient, with an optimal scaling in $p$ for the efficiency
constant. It is therefore true that our new estimators do not seem to present
significant improvements over the standard ones as far as efficiency is concerned.
However, these standard estimators are not suited for adaptivity purposes, and
adaptive algorithms rather use residual-based estimators similar to the one we propose
in~\eqref{eq_residual_based_intro} that were previously established
in~\cite[Theorem 3.1]{alonso_1996a} and \cite[Lemma 5.1]{carstensen_1997a}.
Unfortunately, both works are limited to two-dimensional problems and trivial topologies.
More specifically, neither work consider mixed boundary conditions, and the domain
is respectively assumed to be simply--connected and convex in
\cite{alonso_1996a} and \cite{carstensen_1997a}. Besides, the estimates obtained
in \cite{alonso_1996a,carstensen_1997a} are not $hp$-optimal.
As previously mentioned, to the best of the author's knowledge, all works dealing with
(optimal) convergence of adaptive mixed finite element discrization either
explicitly requires a trivial topology, as in~\cite{cascon_nochetto_sibert_2007a,chen_holst_xu_2008a},
or are based on~\cite{alonso_1996a,carstensen_1997a} where the assumption is
made as in~\cite{becker_mao_2008a,carstensen_hoppe_2006a,carstensen_rabus_2011a}.
The present results thus allow the above mentioned adaptive convergence results
to apply to general domains and boundary conditions.

(c) Finally, for discontinuous Galerkin discretizations,
an estimation such as~\eqref{eq_poincare_intro} does not hold,
unless the stabilization penalization is artificially
enlarged, see~\cite{ainsworth_rankin_2010a,bonito_nochetto_2010a,karakashian_pascal_2007a}.
As a result, error estimators including the jumps of the potential are not shown
to be efficient in general, unless said jumps are included in the error measure, see
e.g.~\cite{ainsworth_fu_2018a,
chen_li_xiang_2021a,
ern_vohralik_2015a,
houston_schotzau_wihler_2006a}.
This lack of efficiency in a natural error measure further causes problems in
the analysis of adaptive schemes based on the estimator, as in~\cite{bonito_nochetto_2010a,
gudi_guzman_2014a,
karakashian_pascal_2003a,
karakashian_pascal_2007a,
kreuzer_georgoulis_2018a}.
To the best of our knowledge, the estimators proposed in this work are the first
to be efficient for discontinuous Galerkin discretizations in a natural norm
without any assumption on the stabilization.

We finally mention two ideas related to the one developed here that recently appeared
in the literature. (i) Closest to the present results are the previous
work~\cite{chaumontfrelet_2025a} of the author and~\cite{bertrand_carstensen_grassle_tran_2023a}.
In a nutshell, the key idea of both works is to employ extra orthogonality conditions satisfied
by $\GRAD_h$ to replace the gradient vector field by a curl-free one
in~\eqref{eq_prager_synge_intro}. However, the setting in~\cite{chaumontfrelet_2025a}
is fairly different from here, since it considers magnetostatic problems discretized
by mixed methods, and only covers flux-equilibrated estimators. On the other hand,
whereas the present results are very closely related
to~\cite{bertrand_carstensen_grassle_tran_2023a}, they represent a significant improvement.
First, the technique used in~\cite{bertrand_carstensen_grassle_tran_2023a}
leads to unknown constants depending on the domain. This can be specifically
seen in~\cite[Theorem 1]{bertrand_carstensen_grassle_tran_2023a}. Besides, although
the estimator proposed in~\cite[Equation (4)]{bertrand_carstensen_grassle_tran_2023a}
is suited for lowest-order discretization, it is not $hp$-optimal. In contrast,
the residual-based estimator proposed Lemma~\ref{lemma_res_abstract_rel} below
has an optimal $hp$-scaling. (ii) As already alluded to earlier, another displeasing
aspect of the jump term in~\eqref{eq_jump_intro} is the asymmetry of the $p$ and $h_K$
powers in the multiplicative factor, which suggests an $hp$-suboptimality. This is a
well-known problem that stems from the (apparently necessary) use of an inverse inequality.
In~\cite{cangiani_dong_georgoulis_2023a,dong_ern_2024a},
ideas related to the one employed below and
in~\cite{bertrand_carstensen_grassle_tran_2023a}
are used to bypass such an inverse inequality. Yet, the focus is different,
since the final error estimators in~\cite{cangiani_dong_georgoulis_2023a,dong_ern_2024a}
still contains a jump term from the potential (albeit with an improved $p$-scaling),
and is therefore not efficient in the natural energy norm in general.

The remainder of this work is organized as follows.
In Section~\ref{section_setting}, we make the setting precise,
introduce key notation and state useful results from the literature.
Section~\ref{section_abstract} is dedicated to the abstract derivation
of~\eqref{eq_new_identities} under generic assumptions on $\GRAD_h$.
Finally, Sections~\ref{section_residual},~\ref{section_equilibrated}
and~\ref{section_alternative} are respectively devoted to the three
practical estimators we construct from~\eqref{eq_new_identities}, namely,
a residual-based estimator, a fully-equilibrated estimator, and an alternative
equilibration strategy.

\section{Setting}
\label{section_setting}

\subsection{Domain}

Throughout this manuscript, $\Omega \subset \mathbb R^3$ is an open domain
with a Lipschitz polyhedral boundary $\partial \Omega$. The boundary is split
into two Lipschitz relatively open polyhedral subdomains $\GD$ and $\GN$
in such a way that $\overline{\GD} \cup \overline{\GN} = \partial \Omega$.
Notice that we do not assume that $\Omega$, nor $\GD$ or $\GN$ are simply connected.
We denote by $\bn$ the unit vector normal to $\partial \Omega$ pointing outside
$\Omega$.

\subsection{Diffusion coefficient}

The diffusion coefficient
$\AAA: \Omega \to \mathbb R^{3 \times 3}$ is a measurable
symmetric matrix-valued function. For any measurable set $U \subset \Omega$,
we set
\begin{equation*}
\alpha_{\min,U} \eq \underset{\bx \in U}{\operatorname{ess\;inf}}
\min_{\substack{\be \in \mathbb R^d \\ |\be| = 1}}
\AAA(\bx) \be \cdot \be,
\qquad
\alpha_{\max,U} \eq \underset{\bx \in U}{\operatorname{ess\;sup}}
\max_{\substack{\be \in \mathbb R^d \\ |\be| = 1}}
\AAA(\bx) \be \cdot \be,
\end{equation*}
and we assume that $\alpha_{\min,\Omega} > 0$ and $\alpha_{\max,\Omega} < +\infty$.
Some results below will require that $\AAA$ is piecewise polynomial, which we will
specify case by case.

\subsection{Function spaces}

For any open set $U \subset \Omega$ with Lipschitz boundary, $L^2(U)$ is the Lebesgue space of
square-integrable functions and $\BL^2(U) \eq [L^2(U)]^3$. The usual norm and
inner product of both spaces are denoted by $\|{\cdot}\|_U$ and $({\cdot},{\cdot})_U$.
If $\underline{\BM}: U \to \mathbb R^{3 \times 3}$ is a measurable symmetric
matrix-valued functions, we employ the notation $\|\bv\|_{\underline{\BM},U}^2
\eq (\underline{\BM}\bv,\bv)_U$ for $\bv \in \BL^2(U)$. If $\underline{\BM}$
is uniformly bounded from above and below in the sense of quadratic forms,
then $\|{\cdot}\|_{\underline{\BM},U}$
is a norm on $\BL^2(U)$ equivalent to standard one. We will also employ the same notation
for Lebesgue spaces on two-dimensional relatively open affine subsets $V \subset \Omega$.

The standard Sobolev space of scalar functions $v \in L^2(U)$ with $\grad v \in \BL^2(U)$
is denoted by $H^1(U)$, and we let $\BH^1(U) \eq [H^1(U)]^3$. If $\gamma \subset \partial U$
is a relatively open subset, then $H^1_{\gamma}(U) \subset H^1(U)$ collects functions with
vanishing trace on $\gamma$. For $\bv \in \BH^1(U)$, we denote by
$\tgrad \bv = (\grad \bv_1,\cdots,\grad \bv_d)$ its gradient, and the application 
\begin{equation*}
\|\tgrad \bv\|_{\underline{\BM},\Omega}^2 \eq \sum_{k=1}^d \|\grad \bv_k\|_{\underline{\BM},\Omega}^2,
\end{equation*}
is then equivalent to the standard semi-norm of $\BH^1(\Omega)$ if $\underline{\BM}: U \to \mathbb R^{3 \times 3}$
is symmetric, measurable and uniformly bounded from abvoe and below. We skip the $\underline{\BM}$
subscript when $\underline{\BM} = \underline{\BI}$. Lebesgue and Sobolev spaces are
extensively described, e.g. in~\cite{adams_fournier_2003a}.

If $v: U \to \mathbb{R}$ is measurable, we will employ the notation
\begin{equation*}
\|v\|_{L^\infty(U)} \eq \operatorname{ess} \sup_{\bx \in U} |v(\bx)|
\end{equation*}
for its essential supremum. Similarly, if $\bv: U \to \mathbb{R}^3$,
then $\|\bv\|_{\BL^\infty(U)} = \||\bv|\|_{L^\infty(\Omega)}$.

We will also use Sobolev spaces of vector-valued functions with well-defined curl
and divergence. Namely, $\BH(\ccurl,U)$ and $\BH(\ddiv,U)$ respectively contain
functions $\bv \in \BL^2(U)$ such that $\curl \bv \in \BL^2(U)$ and $\div \bv \in L^2(U)$,
where $\curl$ and $\div$ are the curl and divergence operators defined in the sense
of distributions. We further denote by $\BH_{\gamma}(\ccurl,U)$ and $\BH_{\gamma}(\ddiv,U)$
the subspaces of functions with vanishing tangential and normal traces on
$\gamma$, respectively. For
more details, we refer the reader to~\cite{girault_raviart_1986a}, and in particular
to~\cite{alonso_valli_1996a} for the definition of traces on portions of the boundary.

We use the short-hand notation $L^2_\gamma(U) \eq \div \BH_\gamma(\ddiv,U)$.
This space coincides with $L^2(U)$ unless $\gamma = \partial U$, in which case it is
the subset of functions with vanishing mean-value. Finally, $\BH_\gamma(\ddiv^0,U)$
is the subspace of $\BH_{\gamma}(\ddiv,U)$ collecting divergence-free functions.

\subsection{Mesh}

We consider a simplicial mesh $\CT_h$ of $\Omega$ consisting of (open) tetrahedral elements $K$.
We assume that the mesh is conforming, i.e. $\cup_{K \in \CT_h} \overline{K} = \overline{\Omega}$,
and matching, meaning that if the intersection $\overline{K_-} \cap \overline{K_+}$ of two
distinct elements $K_\pm \in \CT_h$ is non-empty, then it is a subsimplex (i.e. a vertex,
a full edge or a full face) of both elements. These assumptions are fully standard,
see e.g.~\cite[Definition 8.11]{ern_guermond_2021a}, but exclude hanging-nodes.
We then respectively denote by $\CV_h$, $\CE_h$ and $\CF_h$ the set of vertices,
edges and faces of the mesh $\CT_h$, and we
classically assume that the mesh matches the boundary decomposition. This means
that for all faces $F \in\CF_h$, if $F \subset \partial \Omega$, then either
$F \subset \GD$ or $F \subset \GN$.

For an element $K \in \CT_h$, $\CV(K) \subset \CV_h$, $\CE(K) \subset \CE_h$, and
$\CF(K) \subset \CF_h$ are the set of vertices, edges and faces of $K$. We also
introduce the subsets $\CFIh, \CFDh, \CFNh \subset \CF_h$ of faces
respectively
lying inside $\Omega$, on $\GD$ and on $\GN$. We associate with each face $F \in \CF_h$ a unit
normal vector $\bn_F$.  When $F \in \CFDh \cup \CFNh$, we require that $\bn_F$ points outside
$\Omega$, whereas the orientation is arbitrary but fixed if $F \in \CFIh$.

For all $K \in \CT_h$, $h_K$ is the diameter of $K$ and $\rho_K$ is the diameter of
the largest ball contained in $\overline{K}$. The value
$\kappa_K \eq h_K/\rho_K \geq 1$ is then
the shape-regularity parameter of $K$. For any submesh $\CT \subset \CT_h$,
we let $\kappa_{\CT} \eq \max_{K \in \CT} \kappa_K$. Finally, if $U \subset \Omega$
is a subset covering the elements $K \in \CT$, we will also alternatively employ
the notation $\kappa_U \eq \kappa_{\CT}$.

\subsection{Element patches}

For $\ba \in \CV_h$, we denote by $\CTa \subset \CT_h$ the vertex patch
of elements $K \in \CT_h$ such that $\ba \in \CV(K)$, and by $\CVa$ the
corresponding set of vertices. We also employ the
notation $\oma \subset \Omega$ for the associated open domain.
We further employ the notation $\tCTa = \cup_{\bb \in \CVa} \CTa$
for the extended vertex patch including one extra later of elements,
and denote by $\toma \subset \Omega$ the corresponding open domain.

For $K \in \CT_h$, the element patch $\omK \subset \Omega$ is the open domain
covering all the elements $K' \in \CT_h$ sharing a vertex with $K$. We also
respectively denote by $\tomK$ and $\ttomK$ the element patches extended by
one and two layers of elements.

\subsection{Jumps and averages}

Let $H^1(\CT_h)$ stand for the set of $v \in L^2(\Omega)$ such that
$v|_K \in H^1(K)$ for all $K \in \CT_h$. If $F \in \CFIh$ is an internal
face, we introduce the jump and average of a function $v \in H^1(\CT_h)$
through $F$ as follows
\begin{equation*}
\jmp{v}_F \eq v_+ \bn_+ \cdot \bn_F + v_- \bn_- \cdot \bn_F
\qquad
\avg{v}_F \eq \frac{1}{2}(v_+ + v_-),
\end{equation*}
where $K_\pm \in \CT_h$ are the only two elements such that $F = \partial K_- \cap \partial K_+$,
$v_\pm \eq v|_{K_\pm}|_F$, and $\bn_\pm$ is the unit normal pointing outside $K_\pm$. On the other hand,
for a boundary face $F \in \CFDh \cup \CFNh$, we simply set
\begin{equation*}
\jmp{v}_F \eq v|_F \qquad \avg{v}_F \eq v|_F.
\end{equation*}
We omit the subscript $F$ when it is clear from context, and also employ similar notation
for vector-valued functions.

\subsection{Finite element spaces}

For $q \geq 0$ and $K \in \CT_h$, $\CP_q(K)$ is the space of polynomials of degree
at most $q$ defined on $K$, $\BCP_q(K) \eq [\CP_q(K)]^3$, and $\TCP_q(K) \eq [\CP_q(K)]^{3 \times 3}$.
For $q \geq 1$, we will also need the Raviart--Thomas and N\'ed\'elec polynomial spaces
$\RT_q(K) \eq \BCP_{q-1}(K) + \bx \CP_{q-1}(K)$ and
$\ND_q(K) \eq \BCP_{q-1}(K) + \bx \times \BCP_{q-1}(K)$.
These spaces are extensively discussed in~\cite[Chapters 7, 14 and 15]{ern_guermond_2021a}.
When $\CT \subset \CT_h$, $\CP_q(\CT)$ stands for the set of functions whose restriction
to each $K \in \CT$ belongs to $\CP_q(K)$, and we employ similar definitions for
$\BCP_q(\CT)$ $\TCP_q(\CT)$, $\RT_q(\CT)$ and $\ND_q(\CT)$. Notice that these spaces do not have
build-in compatibility conditions.

For a face $F \in \CF_h$ and $q \geq 0$, $\CP_q(F)$ denotes the set of polynomials of
degree at most $q$ defined on $F$, and if $v_q \in \CP_q(F)$ then
$\grad_F v_q$ is its tangential gradient along $F$.

\subsection{Data-oscillation}

For $v \in L^2(\Omega)$ and $q \geq 1$, we employ the notation
\begin{equation}
\label{eq_data_oscillation}
\Pi_{hq} v \eq \arg \min_{v_h \in \CP_{q-1}(\CT_h)} \|v-v_h\|_\Omega.
\end{equation}
This projection is local, and we have
\begin{equation}
\label{eq_data_oscillation_h}
\|v-\Pi_{hq} v\|_K \leq \frac{h_K}{\pi} \|\grad v\|_K
\end{equation}
as well as
\begin{equation}
\label{eq_data_oscillation_hp}
\|v-\Pi_{hq} v\|_K \leq C(\kappa_K) \frac{h_K}{q} \|\grad v\|_K
\end{equation}
for all $K \in \CT_h$. These estimates are respectively established
in~\cite{bebendorf_2003a} and~\cite{babuska_suri_1987a}.

\subsection{Cohomology space}

We have the orthogonal decomposition
\begin{equation*}
\label{eq_cohomology}
\BH_{\GN}(\ddiv^0,\Omega)
=
\curl \BH_{\GN}(\ccurl,\Omega) \oplus_{\AAA^{-1}} \LBA
\end{equation*}
in the inner product $(\AAA^{-1}\cdot,\cdot)_\Omega$,
where $\LBA \subset \BH_{\GN}(\ddiv^0)$ is a finite-dimensional
subspace, see~\cite{duff_1952a}.

We introduce
\begin{equation}
\label{eq_gamma}
\gamma_h
\eq
\max_{\substack{\lba \in \LBA \\ \|\lba\|_{\AAA^{-1},\Omega} = 1}}
\min_{\substack{\lba_h \in \RT_1(\CT_h) \cap \BH_{\GN}(\ddiv^0,\Omega)
\\
\lba-\lba_h \in \curl\BH_{\GN}(\ccurl,\Omega)}}
\|\lba-\lba_h\|_{\AAA^{-1},\Omega}.
\end{equation}
We will establish the following result in Appendix~\ref{appendix_gamma}.

\begin{lemma}[Control of $\gamma_h$]
\label{lemma_gamma}
We have
\begin{equation}
\label{eq_gamma_st}
\gamma_h
\leq
C(\kappa_{\CT_h}) \sqrt{\max_{K \in \CT_h}\frac{\alpha_{\max,\omK}}{\alpha_{\min,\omK}}}.
\end{equation}
In addition, there exists an $s > 0$ such that
\begin{equation}
\label{eq_gamma_cv}
\gamma_h
\leq
C \left (\frac{h}{\ell_\Omega}\right )^s,
\end{equation}
where $\ell_\Omega$ is the diameter of $\Omega$.
In~\eqref{eq_gamma_cv}, $s$ and $C$ depend on the shape of $\Omega$
and of the boundary partition, as well as on $\AAA$.
\end{lemma}

We note that it is also possible to introduce a definition
of $\gamma_h$ with higher-order elements, whereby convergence
in~\eqref{eq_gamma_cv} also happens when increasing the degree.
For simplicity, we refrain from doing so here.

\subsection{Model problem}

Throughout the manuscript we consider a fixed load term $f: \Omega \to \mathbb R$,
and we look for $u: \Omega \to \mathbb R$ such that
\begin{equation}
\label{eq_poisson_strong}
\left \{
\begin{array}{rcll}
-\div(\AAA\grad u) &=& f & \text{ in } \Omega,
\\
u &=& 0 & \text{ on } \GD,
\\
\AAA \grad u \cdot \bn &=& 0 & \text{ on } \GN.
\end{array}
\right .
\end{equation}
More formally, we assume that $f \in L^2(\Omega)$, and $u$ is the unique element
of $H^1_{\GD}(\Omega)$ such that
\begin{equation}
\label{eq_poisson_weak}
(\AAA\grad u,\grad v)_\Omega = (f,v)_\Omega
\end{equation}
for all $v \in H^1_{\GD}(\Omega)$.

\subsection{Numerical schemes}

We conclude this section with the numerical schemes that we will analyze.

\subsubsection{Discontinuous Galerkin scheme}
\label{section_ipdg_scheme}

For $v_h \in \CP_p(\CT_h)$, we define its jump lifting
$\LIFT_h(v_h) \in \BCP_{p-1}(\CT_h)$ by requiring that
\begin{equation*}
(\LIFT_h(v_h),\bw_h)_\Omega = \sum_{F \in \CFIh \cup \CFDh}
(\jmp{v_h},\avg{\bw_h} \cdot \bn_F)_F
\end{equation*}
for all $\bw_h \in \BCP_{p-1}(\CT_h)$. We then introduce
the discrete gradient of $v_h \in \CP_p(\CT_h)$ by setting
\begin{equation*}
\GRAD_h(v_h) \eq \grad_h v_h + \LIFT_h(v_h).
\end{equation*}
One easily checks (see also~\cite[Section 3]{chaumontfrelet_2023a}) that
\begin{equation}
\label{eq_stokes_discrete_gradient}
(\GRAD_h(v_h),\bxi_h)_\Omega = (v_h,\div \bv_h)_\Omega
\end{equation}
for all $\bxi_h \in \BCP_{p-1}(\CT_h) \cap \BH_{\GN}(\ddiv,\Omega)$.

Consider a symmetric bilinear stabilization form
$s_h: \CP_p(\CT_h) \times \CP_p(\CT_h) \to \mathbb R$
such that $s_h(v_h,w_h) = 0$ for all $v_h \in \CP_p(\CT_h) \cap H^1_{\GD}(\Omega)$
and $w_h \in \CP_p(\CT_h)$, and, for all $z_h \in \CP_p(\CT_h)$, $s_h(z_h,z_h) > 0$
and $s_h(z_h,z_h) = 0$ implies that $z_h \in H^1_{\GD}(\Omega)$. Then,
the interior penalty discontinuous Galerkin (IPDG) schemes reads as follows: There exists a unique $u_h \in \CP_p(\CT_h)$
such that
\begin{equation}
\label{eq_ipdg_scheme}
(\AAA\GRAD_h(u_h),\GRAD_h(v_h))_\Omega + s_h(u_h,v_h) = (f,v_h)_\Omega
\end{equation}
for all $v_h \in \CP_p(\CT_h)$.

For later references, we observe that
\begin{equation}
\label{eq_ipdg_galerkin_orthogonality}
(\AAA \GRAD_h(u_h),\grad v_h)_\Omega = (f,v_h)_\Omega
\end{equation}
for all $v_h \in \CP_p(\CT_h) \cap H^1_{\GD}(\Omega)$,
since the stabilization vanishes and $\GRAD_h(v_h) = \grad v_h$ for conforming functions.

\begin{remark}[Stability]
In practice, one needs to require that
\begin{equation}
\label{eq_assumption_stabilization}
s(v_h,v_h) \geq C(\kappa_{\CT_h}) \sum_{F \in \CFIh \cap \CFDh} \alpha_{\max,\omF} \frac{p^2}{h_F} \|\jmp{u_h}\|_F^2
\end{equation}
to ensure suitable stability properties of the bilinear form in the left-hand
side of~\eqref{eq_ipdg_scheme}, and optimal convergence rates for the approximation
$u_h$ to $u$. However, we do not need such an assumption here from the a posteriori
point of view. In fact, our analysis applies to any function $u_h \in \CP_p(\CT_h)$
such that \eqref{eq_ipdg_galerkin_orthogonality} holds true.
\end{remark}

\begin{remark}[Standard formulation]
\label{remark_standard_IPDG}
The IPDG scheme is typically not formulated (and implemented) as presented
in~\eqref{eq_ipdg_scheme}, but rather, one looks for $u_h \in \CP_p(\CT_h)$
such that $a_h(u_h,v_h) = (f,v_h)_\Omega$ for all $v_h \in \CP_p(\CT_h)$
where
\begin{multline*}
a_h(u_h,v_h)
\eq
(\AAA\grad_h u_h,\grad_h v_h)_\Omega
\\
-
\sum_{F \in \CFIh \cup \CFDh} (\jmp{u_h},\avg{\AAA\grad v_h} \cdot \bn_F)_F
-
\sum_{F \in \CFIh \cup \CFDh} (\avg{\AAA\grad u_h} \cdot \bn_F,v_h)_F
\\
+
\beta \sum_{F \in \CFIh \cup \CFDh} \alpha_{\max,\omF} \frac{p^2}{h_F} (\jmp{u_h},\jmp{v_h})_F
\end{multline*}
with $\beta > 0$ large enough. When $\AAA \in \TCP_0(\CT_h)$,
the definition of the lifting $\LIFT_h$ shows that in fact
\begin{equation*}
a_h(u_h,v_h)
=
(\AAA\GRAD_h(u_h),\GRAD_h(v_h))_\Omega
+
s_h(u_h,v_h)
\end{equation*}
where
\begin{align*}
s_h(u_h,v_h)
\eq
\beta \sum_{F \in \CFIh \cup \CFDh} \alpha_{\max,\omF} \frac{p^2}{h_F} (\jmp{u_h},\jmp{v_h})_F
-
(\AAA\LIFT_h(u_h),\LIFT_h(v_h))_\Omega,
\end{align*}
satisfies~\eqref{eq_assumption_stabilization}.
\end{remark}

\subsubsection{Mixed finite element method}
\label{section_mixed_fem}

The mixed finite element formulation consists in finding
$(\bsig_h,r_h) \in \RT_p(\CT_h) \cap \BH_{\GN}(\ddiv,\Omega) \times \CP_{p-1}(\CT_h) \cap L^2_{\GN}(\Omega)$
such that
\begin{equation}
\label{eq_mixed_fem}
\left \{
\begin{array}{rcl}
(\AAA^{-1}\bsig_h,\bv_h)_\Omega + (r_h,\div \bv_h)_\Omega &=& 0
\\
(\div \bsig_h,q_h)_\Omega &=& (f,q_h)_\Omega
\end{array}
\right .
\end{equation}
for all $\bv_h \in \RT_p(\CT_h) \cap \BH_{\GN}(\ddiv,\Omega)$ and $q_h \in \CP_{p-1}(\CT_h) \cap L^2_{\GN}(\Omega)$.
This saddle-point problem is well-posed due to standard inf-sup theory, see~\cite{brezzi_1974a}.

For later references, we record that the first and second equations in~\eqref{eq_mixed_fem}
respectively imply that
\begin{equation}
\label{eq_orthogonality_mixed_fem}
(\AAA^{-1}\bsig_h,\bxi_h)_\Omega = 0
\end{equation}
for all $\bxi_h \in \RT_p(\CT_h) \cap \BH_{\GN}(\ddiv^0,\Omega)$ and
\begin{equation}
\label{eq_equilibration_mixed_fem}
\div \bsig_h = \Pi_{hp} f.
\end{equation}
These two equations are the key properties needed in our a posteriori analysis.
Crucially, our estimators do not use the Lagrange multiplier $r_h$, and can be
solely computed from $\bsig_h$.

\subsubsection{Other examples of covered schemes}

We finally mention that other types of scheme are covered by our abstract framework,
but that we won't cover those in detail. These include mixed conforming discretizations of
the primal-hybrid formulation of~\eqref{eq_poisson_strong} as described
in~\cite{raviart_thomas_1977b} and non-conforming discretization of~\eqref{eq_poisson_weak} by
Crouzeix--Raviart elements as per~\cite{crouzeix_raviart_1973a,ciarlet_dunkl_sauter_2018a}.

\section{Generalized Prager--Synge estimates}
\label{section_abstract}

As motivated in the introduction, our core contribution is to derive generalized
Prager--Synge estimates. The ``standard'' Prager--Synge identity is established
in~\eqref{eq_standard_prager_synge} below for completeness. It was was first seminally
established in~\cite{prager_synge_1947a} for conforming primal finite element discretizations,
and later extended to non-conforming discretizations, as in~\cite{ern_vohralik_2015a}.
Our goal here is to provide an alternative version, whereby the second term in the right-hand
(which deals with the non-conforming nature of the discretization) is reformulated in a way
that leads to new error estimators.

\subsection{The standard Prager--Synge identity}

For completeness, we start by recapping the standard Prager--Synge identity.
Our proof largely follows~\cite{ern_vohralik_2015a} but is marginally modified,
since the statement we give below is slightly more general than the one given
in~\cite{ern_vohralik_2015a}.

\begin{theorem}[Prager--Synge identities]
For all $\GRAD \in \BL^2(\Omega)$, we have
\begin{equation}
\label{eq_standard_prager_synge}
\|\grad u-\GRAD\|_{\AAA,\Omega}^2
=
\min_{\substack{\bsig \in \BH_{\GN}(\ddiv,\Omega) \\ \div \bsig = f}}
\|\GRAD+\AAA^{-1}\bsig\|_{\AAA,\Omega}^2
+
\min_{s \in H^1_{\GD}(\Omega)} \|\GRAD-\grad s\|_{\AAA,\Omega}^2.
\end{equation}
\end{theorem}

\begin{proof}
We first introduce $s^\star$ as the unique element of $H^1_{\GD}(\Omega)$ such that
\begin{equation}
(\AAA\grad s^\star,\grad v)_\Omega = (\AAA\GRAD,\grad v)_\Omega
\end{equation}
for all $v \in H^1_{\GD}(\Omega)$, so that
\begin{align*}
\|\grad u -\GRAD\|_{\AAA,\Omega}^2
&=
\|\grad u-\grad s^\star\|_{\AAA,\Omega}^2
+
\|\GRAD-\grad s^\star\|_{\AAA,\Omega}^2
\\
&=
\|\grad(u-s^\star)\|_{\AAA,\Omega}^2
+
\min_{s \in H^1_{\GD}(\Omega)} \|\GRAD-\grad s\|_{\AAA,\Omega}^2.
\end{align*}

For all $\bsig \in \BH_{\GN}(\ddiv,\Omega)$ and $v \in H^1_{\GD}(\Omega)$, we have
\begin{equation*}
(\AAA\grad u,\grad v)_\Omega = (f,v)_\Omega = (\div \bsig,v)_\Omega = -(\bsig,\grad v)_\Omega
\end{equation*}
so that
\begin{equation*}
(\AAA \grad(u-s^\star),\grad v)_\Omega = -(\bsig+\AAA\GRAD,\grad v)_\Omega
\end{equation*}
and therefore
\begin{equation*}
\|\grad(u-s^\star)\|_{\AAA,\Omega} \leq \|\AAA^{-1}\bsig+\GRAD\|_{\AAA,\Omega},
\end{equation*}
by picking $v = u-s^\star$. It follows that
\begin{equation}
\label{tmp_min_prager_synge}
\|\grad(u-s^\star)\|_{\AAA,\Omega}
\leq
\min_{\substack{
\bsig \in \BH_{\GN}(\ddiv,\Omega)
\\
\div \bsig = f
}}
\|\AAA^{-1}\bsig+\GRAD\|_{\AAA,\Omega}
\end{equation}
since $\bsig$ was arbitrary. This proves the upper bound in~\eqref{eq_standard_prager_synge}.
For the converse bound, we need to analyze the minimizer $\bsig^\star$ in the right-hand side
of~\eqref{tmp_min_prager_synge}. The Euler-Lagrange equations associated with the minimization
problem are to find $\bsig^\star \in \BH_{\GN}(\ddiv,\Omega)$ and $q \in L^2_{\GN}(\Omega)$
such that
\begin{equation*}
\left \{
\begin{array}{rcl}
(\AAA^{-1}\bsig^\star,\bv)_\Omega + (q,\div \bv)_\Omega &=& -(\GRAD,\bv)_\Omega
\\
(\div \bsig^\star,r)_\Omega &=& (f,r)_\Omega
\end{array}
\right .
\end{equation*}
for all $\bv \in \BH_{\GN}(\ddiv,\Omega)$ and $r \in L^2_{\GN}(\Omega)$.
Using the definition of the weak gradient, we read from the
first equation that $q \in H^1_{\GD}(\Omega)$ with
\begin{equation*}
\grad q = \AAA^{-1}\bsig^\star+\GRAD,
\end{equation*}
and it remains to show that
\begin{equation*}
\|\grad q\|_{\AAA,\Omega}
\leq
\|\grad(u-s^\star)\|_{\AAA,\Omega}.
\end{equation*}
We have
\begin{equation*}
\|\grad q\|_{\AAA,\Omega}^2
=
(\AAA\grad q,\grad q)
=
(\bsig^\star,\grad q)+(\AAA\GRAD,\grad q).
\end{equation*}
On the one hand,
\begin{equation*}
(\bsig^\star,\grad q) = -(\div \bsig^\star,q)_\Omega = -(f,q)_\Omega = -(\AAA \grad u,\grad q)_\Omega
\end{equation*}
whereas on the other hand
\begin{equation*}
(\AAA\GRAD,\grad q) = (\AAA\grad s^\star,\grad q),
\end{equation*}
leading to
\begin{equation*}
\|\grad q\|_{\AAA,\Omega}^2
=
(\AAA\grad (s^\star-u),\grad q)
\leq
\|\grad (s^\star-u)\|_{\AAA,\Omega}
\|\grad q\|_{\AAA,\Omega},
\end{equation*}
which concludes the proof.
\end{proof}

\subsection{Generalization}

We now aim at rewriting the second term in the right-hand side of~\eqref{eq_standard_prager_synge}
in alternative ways. Our first result applies without any assumption on the domain, nor on the
field $\GRAD$.

\begin{lemma}[Divergence--free maximimization]
For all $\GRAD \in \BL^2(\Omega)$, we have
\begin{equation}
\label{eq_divfree_identity}
\min_{s \in H^1_{\GD}(\Omega)}
\|\GRAD-\grad s\|_{\AAA,\Omega}
=
\max_{\substack{\bxi \in \BH_{\GN}(\ddiv^0\Omega) \\ \|\bxi\|_{\AAA^{-1},\Omega} = 1}}
(\GRAD,\bxi)_\Omega.
\end{equation}
\end{lemma}

\begin{proof}
For the sake of simplicity, we introduce
\begin{equation*}
s^\star \eq \arg \min_{s \in H^1_{\GD}(\Omega)} \|\GRAD-\grad s\|_{\AAA,\Omega}.
\end{equation*}
Classically, $s^\star$ is the only element of $H^1_{\GD}(\Omega)$ such that
\begin{equation}
\label{tmp_EL_sstar}
(\AAA\grad s^\star,\grad v)_\Omega = (\AAA\GRAD,\grad v)_\Omega
\end{equation}
for all $v \in H^1_{\GD}(\Omega)$. Let us set $\bsig^\star = \AAA(\grad s^\star-\GRAD)$.
Then, it follows from \eqref{tmp_EL_sstar} that $\bsig^\star \in \BH_{\GN}(\ddiv^0,\Omega)$.

On the one hand, by definition of $\bsig^\star$, we have
\begin{equation*}
\|\bsig^\star\|_{\AAA^{-1},\Omega}^2
=
(\AAA^{-1}\bsig^\star,\bsig^\star)_\Omega
=
(\grad s^\star-\GRAD,\AAA(\grad s^\star-\GRAD))_\Omega
=
\|\grad s^\star-\GRAD\|_{\AAA,\Omega}^2.
\end{equation*}
On the other hand, we have
\begin{equation*}
\|\bsig^\star\|_{\AAA^{-1},\Omega}
=
\sup_{\substack{\bxi \in \BH_{\GN}(\ddiv^0,\Omega) \\ \|\bxi\|_{\AAA^{-1},\Omega} = 1}}
(\AAA^{-1}\bsig^\star,\bxi)_\Omega
=
\sup_{\substack{\bxi \in \BH_{\GN}(\ddiv^0,\Omega) \\ \|\bxi\|_{\AAA^{-1},\Omega} = 1}}
(\grad s^\star-\GRAD,\bxi)_\Omega
=
\sup_{\substack{\bxi \in \BH_{\GN}(\ddiv^0,\Omega) \\ \|\bxi\|_{\AAA^{-1},\Omega} = 1}}
(\GRAD,-\bxi)_\Omega,
\end{equation*}
since $(\grad s^\star,\bxi)_\Omega = (s^\star,\div \bxi)_\Omega = 0$.
This shows \eqref{eq_divfree_identity} with a ``$\sup$'' instead of a ``$\max$''.
To show that the supremum is indeed achieved, we simply observe that the maximizer
is $\bxi = \bsig^\star/\|\bsig^\star\|_{\AAA^{-1},\Omega}$.
\end{proof}

The formulation in~\eqref{eq_divfree_identity} above is still not satisfactory,
since working with divergence-free functions is not convenient. We would rather
work with the curl of vector potentials. However, these two notions are only
equivalent if the topology of the domain is trivial. Here, we show that we
can still resort to vector potentials for general domains if we assume some
orthogonality conditions for the field $\GRAD$. Similar ideas have been used
in~\cite[Lemma 2]{bertrand_carstensen_grassle_tran_2023a}.

\begin{theorem}[Alternative reformulations]
\label{theorem_curl_identity}
Assume that $\GRAD \in \BL^2(\Omega)$ satisfies $(\GRAD,\bxi_h)_\Omega = 0$
for all $\bxi_h \in \RT_1(\CT_h) \cap \BH_{\GN}(\ddiv^0,\Omega)$. Then,
\begin{equation}
\label{eq_curl_bound}
\min_{s \in H^1_{\GD}(\Omega)}
\|\GRAD-\grad s\|_{\AAA,\Omega}
\leq (1+\gamma_h^2)^{1/2}
\max_{\substack{\bthe \in \BH_{\GN}(\ccurl,\Omega) \\ \|\curl \bthe\|_{\AAA^{-1},\Omega} = 1}}
(\GRAD,\curl \bthe)_\Omega.
\end{equation}
Besides, we have
\begin{equation}
\label{eq_curl_identity}
\max_{\substack{\bthe \in \BH_{\GN}(\ccurl,\Omega) \\ \|\curl \bthe\|_{\AAA^{-1},\Omega} = 1}}
(\GRAD,\curl \bthe)_\Omega
=
\min_{\bphi \in \BH_{\GD}(\ccurl^0,\Omega)} \|\GRAD-\bphi\|_{\AAA,\Omega}.
\end{equation}
\end{theorem}

\begin{remark}[Trivial topology]
In a trivial topology, \eqref{eq_curl_bound} holds with an equal sign
and $\gamma_h = 0$. This simply follows from
the identities $\grad H^1_{\GD}(\Omega) = \BH_{\GD}(\ccurl^0,\Omega)$
and $\curl \BH_{\GN}(\ccurl,\Omega) = \BH_{\GN}(\ddiv^0,\Omega)$ combined with
\eqref{eq_divfree_identity}. In this case, \eqref{eq_curl_bound} even
holds true without the orthogonality assumption on $\GRAD$.
\end{remark}

\begin{proof}
Consider an arbitrary $\bxi \in \BH_{\GN}(\ddiv^0,\Omega)$.
We employ the orthogonal decomposition
\begin{equation*}
\bxi = \curl \bthe + \lba
\end{equation*}
with $\bthe \in \BH_{\GN}(\ccurl,\Omega)$ and $\lba \in\LBA$
from~\eqref{eq_cohomology}. We also denote by
$\lba_h \in \RT_1(\CT_h) \cap \BH_{\GN}(\ddiv^0,\Omega)$
the minimizer in the definition of $\gamma_h$ in~\eqref{eq_gamma}.
Then, there exists $\bthe^h \in \BH_{\GN}(\ccurl,\Omega)$ such that
\begin{equation*}
\bxi = \curl \bthe^h + \lba_h,
\end{equation*}
and we have
\begin{align*}
\|\curl \bthe^h\|_{\AAA^{-1},\Omega}
&\leq
\|\curl \bthe\|_{\AAA^{-1},\Omega}
+
\|\curl (\bthe-\bthe^h)\|_{\AAA^{-1},\Omega}
\\
&=
\|\curl \bthe\|_{\AAA^{-1},\Omega}
+
\|\lba-\lba_h\|_{\AAA^{-1},\Omega}
\\
&\leq
\|\curl \bthe\|_{\AAA^{-1},\Omega}
+
\gamma_h
\|\lba\|_{\AAA^{-1},\Omega}
\\
&\leq
\sqrt{1+\gamma_h^2}
\|\bxi\|_{\AAA^{-1},\Omega},
\end{align*}
since the decomposition of $\bxi$ is orthogonal.

Recalling our assumption on $\GRAD$, we have
\begin{align*}
(\GRAD,\bxi)_\Omega = (\GRAD,\curl \bthe^h)_\Omega
&
\leq
\max_{\substack{\bthe \in \BH_{\GN}(\ccurl,\Omega) \\ \|\curl \bthe\|_{\AAA^{-1},\Omega} = 1}}
(\GRAD,\curl \bthe)_\Omega
\|\curl \bthe^h\|_{\AAA^{-1},\Omega}
\\
&\leq
\sqrt{1+\gamma_h^2}
\max_{\substack{\bthe \in \BH_{\GN}(\ccurl,\Omega) \\ \|\curl \bthe\|_{\AAA^{-1},\Omega} = 1}}
(\GRAD,\curl \bthe)_\Omega
\|\bxi\|_{\AAA^{-1},\Omega},
\end{align*}
and since $\bxi$ was arbitrary,~\eqref{eq_curl_bound} follows.

The identity in~\eqref{eq_curl_identity} follows from a usual duality principle. We first
note that the case where either term vanishes simply follows from integration
by parts. We can therefore assume without loss of generality that both terms
are positive. We then consider any projection $\bThe^\star \in \BH_{\GN}(\ccurl,\Omega)$
such that
\begin{equation*}
(\AAA^{-1}\curl \bThe^\star,\curl \bv)_\Omega
=
(\AAA^{-1}(\AAA\GRAD),\curl \bv)_\Omega
\end{equation*}
for all $\bv \in \BH_{\GN}(\ccurl,\Omega)$. Such a $\bThe^\star$ is defined
up to a curl-free vector field that could be fixed by a gauge condition.
The Cauchy--Schwarz inequality ensures that
\begin{equation*}
\max_{\substack{\bthe \in \BH_{\GN}(\ccurl,\Omega) \\ \|\curl \bthe\|_{\AAA^{-1},\Omega} = 1}}
(\AAA^{-1}\curl \bThe^\star,\curl \bthe)_\Omega
=
(\AAA^{-1}\curl \bThe^\star,\curl \bthe^\star)_\Omega
=
\|\bThe^\star\|_{\AAA^{-1},\Omega},
\end{equation*}
with $\bthe^\star \eq \bThe^\star/\|\curl \bThe^\star\|_{\AAA^{-1},\Omega}$.
Note that $\curl \bThe^\star \neq \bzero$ by assumption here, so that the
definition of $\bthe^\star$ is sound. We then conclude that
\begin{equation*}
\max_{\substack{\bthe \in \BH_{\GN}(\ccurl,\Omega) \\ \|\curl \bthe\|_{\AAA^{-1},\Omega} = 1}}
(\GRAD,\curl \bthe)_\Omega
=
(\GRAD,\curl \bthe^\star)_\Omega
=
(\GRAD-\bphi,\curl \bthe^\star)_\Omega,
\leq
\|\AAA(\GRAD-\bphi)\|_{\AAA^{-1},\Omega}\|\curl \bthe^\star\|_{\AAA^{-1},\Omega}
\end{equation*}
and in particular that
\begin{equation*}
\max_{\substack{\bthe \in \BH_{\GN}(\ccurl,\Omega) \\ \|\curl \bthe\|_{\AAA^{-1},\Omega} = 1}}
(\GRAD,\curl \bthe)_\Omega
\leq
\|\GRAD-\bphi\|_{\AAA,\Omega}
\end{equation*}
for all $\bphi \in \BH_{\GD}(\ccurl^0,\Omega)$,
since $\|\curl \bthe^\star\|_{\AAA^{-1},\Omega} = 1$.
The converse inequality is trivial since
$\grad H^1_{\GD}(\Omega) \subset \BH_{\GD}(\ccurl^0,\Omega)$.
This proves~\eqref{eq_curl_identity}.
\end{proof}

As a direct consequence of the above reformulations, we can
recast the Prager--Synge identity in~\eqref{eq_standard_prager_synge}
into two alternative formulations that are respectively suited for
residual-based and equilibrated error estimators.

\begin{corollary}[Generalized Prager--Synge estimates]
\label{corollary_generalized_prager_synge}
For all $\GRAD \in \BL^2(\Omega)$ such that $(\GRAD,\bxi_h)_\Omega = 0$
for all $\bxi_h \in \RT_1(\CT_h) \cap \BH_{\GN}(\ddiv^0,\Omega)$, we have
\begin{subequations}
\begin{multline}
\label{eq_generalized_prager_synge}
\|\grad u-\GRAD\|_{\AAA,\Omega}^2
\leq
\max_{\substack{v \in H^1_{\GD}(\Omega) \\ \|\grad v\|_{\AAA,\Omega} = 1}}
|(f,v)_\Omega-(\AAA\GRAD,\grad v)_\Omega|^2
+
\\
(1+\gamma_h^2)^{1/2}
\max_{\substack{\bthe \in \BH_{\GN}(\ccurl,\Omega) \\ \|\curl \bthe\|_{\AAA^{-1},\Omega} = 1}}
|(\GRAD,\curl \bthe)_\Omega|^2
\end{multline}
and
\begin{multline}
\label{eq_generalized_prager_synge_equilibrated}
\|\grad u-\GRAD\|_{\AAA,\Omega}^2
\leq
\min_{\substack{\bsig \in \BH_{\GN}(\ddiv,\Omega) \\ \div \bsig = f}}
\|\GRAD+\AAA^{-1}\bsig\|_{\AAA,\Omega}^2
\\
+
(1+\gamma_h^2)^{1/2}
\min_{\bphi \in \BH_{\GD}(\ccurl^0,\Omega)} \|\GRAD-\bphi\|_{\AAA,\Omega}^2.
\end{multline}
\end{subequations}
\end{corollary}

\section{Residual-based estimators}
\label{section_residual}

In this section, we first leverage the generalized Prager--Synge estimate
established in~\eqref{eq_generalized_prager_synge} to design residual-based
error estimators.

\subsection{Tools}

We first review a set of standard tools from the literature that are instrumental
for the analysis of residual-based estimators.

\subsubsection{Face patches}

The following notation for patches of elements will be useful.
If $F = \partial K_- \cap \partial K_+ \in \CFIh$, then $\omF$
is the open set such that $\overline{\omF} = \overline{K_-} \cup \overline{K_+}$,
and similarly, we let $\omF \eq K$ for an exterior face $F \in \CF_h$
with $F \subset \partial K \cap \partial \Omega$. For $K \in \CT_h$,
we then employ the notation $\omKF \eq \cup_{F \in \CF(K)} \omF$.

\subsubsection{Bubble functions}
\label{section_bubble}

We will also use ``bubble functions'' in the efficiency proofs.
This is a standard technique (see~\cite{ainsworth_oden_2000a,verfurth_2013a}),
and here, we use the polynomial-degree-explicit bounds from~\cite{melenk_wohlmuth_2001a}.
The results in~\cite{melenk_wohlmuth_2001a} are explicitly established
for two-dimensional problems, but the results extend without change for
three-dimensional configurations, see e.g.~\cite[Theorem 4.1.2]{dorfler_sauter_2013a}.

For all $K \in \CT_h$ there exists a ``bubble'' function $b_K: K \to \mathbb{R}$
such that for all $q \geq 1$ and $v_q \in \CP_q(K)$, we have
$b_K v_q = 0$ on $\partial K$,
\begin{subequations}
\label{eq_bubble_K_scalar}
\begin{equation}
\|v_q\|_K \leq C(\kappa_{K}) q \|b_K^{1/2}v_q\|_K
\end{equation}
and
\begin{equation}
\|\grad (b_Kv_q)\|_K \leq C(\kappa_K)\frac{q}{h_K} \|b_K^{1/2}v_q\|_K,
\end{equation}
see \cite[Theorem 2.5]{melenk_wohlmuth_2001a}.
\end{subequations}
Similarly, for all $F \in \CF_h$, there exists a function
$b_F: F \to \mathbb{R}$ such that for all $q \geq 1$ and $w_q \in \CP_q(F)$
\begin{equation}
\label{eq_bubble_F_scalar}
\|w_q\|_F \leq C(\kappa_{\omF}) q \|b_F^{1/2}w_q\|_F.
\end{equation}
In addition, there exists an extension operator
$\LE_F: \CP_q(F) \to H^{1}(\omF)$
such that $\LE_F(w_q) = b_F w_q$ on $F$,
$\LE_F(w_q) = 0$ on $\partial \omF \setminus F$, and
\begin{equation}
\label{eq_stability_extension_scalar}
q h_F^{-1/2} \|\LE_F(w_q)\|_{\omega_F}
+
q^{-1} h_F^{1/2}\|\grad \LE_F(w_q)\|_{\omega_F}
\leq
C(\kappa_{\omF}) \|b_F^{1/2}w_q\|_F
\end{equation}
for all $w_q \in \CP_q(F)$. These results follow from applying
\cite[Lemma 2.6]{melenk_wohlmuth_2001a} with $\alpha = 1$ and
$\varepsilon = q^{-2}$.

For vector-valued functions, we can apply the extension operator
and the above inequalities component-wise. As a result,
for all $\bv_q \in \BCP_q(K)$, we have $b_K \bv_q = \bzero$
on $\partial K$,
\begin{subequations}
\label{eq_bubble_K}
\begin{equation}
\|\bv_q\|_K \leq C(\kappa_K) q \|b_K^{1/2}\bv_q\|_K
\end{equation}
and
\begin{equation}
\|\curl (b_K\bv_q)\|_K \leq C(\kappa_K)\frac{q}{h_K} \|b_K^{1/2}\bv_q\|_K.
\end{equation}
\end{subequations}
Similarly, for all $\bw_q \in \BCP_q(F)$, we have
\begin{equation}
\label{eq_bubble_F}
\|\bw_q\|_F \leq C(\kappa_{\omF}) q \|b_F^{1/2}\bw_q\|_F
\end{equation}
and $\LE_F(\bw_q) \times \bn_F = b_F \bw_q \times \bn_F$ on $F$,
$\LE_F(\bw_q) \times \bn_{\omF} = \bzero$ on $\partial \omF \setminus F$, with
\begin{equation}
\label{eq_stability_extension}
q h_F^{-1/2} \|\LE_F(\bw_q)\|_{\omega_F}
+
q^{-1} h_F^{1/2}\|\curl \LE_F(\bw_q)\|_{\omega_F}
\leq
C(\kappa_{\omF}) \|b_F^{1/2}\bw_q\|_F.
\end{equation}

\subsubsection{Quasi-interpolation operator}
\label{section_quasi_interpolation}

Next, we will need suitable interpolation operators for the reliability
estimate. Following \cite[Theorem 3.3]{karkulik_melenk_2015a},
for all $q \geq 1$, there exists a quasi-interpolation operator
$\CI_{hq}: H^1_{\GD}(\Omega) \to \CP_q(\CT_h) \cap H^1_{\GD}(\Omega)$ such that
\begin{equation}
\label{eq_quasi_interpolation_scalar}
\frac{q}{h_K}
\|v-\CI_{hq}v\|_K
+
\sqrt{\frac{q}{h_K}} \|v-\CI_{hq}v\|_{\partial K}
\leq
C(\kappa_{\omK})
\left (
\ell_\Omega^{-1}\|v\|_{\omK}
+
\|\grad v\|_{\omK}
\right )
\end{equation}
for all $v \in H^1_{\GD}(\Omega)$, where $\ell_\Omega$ denotes the diameter of $\Omega$.
Similarly, for vector-valued functions, we can construct
$\CJ_{hq}: \BH^1_{\GN}(\Omega) \to \ND_q(\CT_h) \cap \BH_{\GN}(\ccurl,\Omega)$
such that
\begin{equation}
\label{eq_quasi_interpolation}
\frac{q}{h_K}
\|\bpsi-\CJ_{hq}\bpsi\|_K
+
\sqrt{\frac{q}{h_K}} \|(\bpsi-\CJ_{hq}\bpsi) \times \bn_K\|_{\partial K}
\leq
C(\kappa_{\omK})
\left (
\ell_\Omega^{-1}\|\bpsi\|_{\omK}
+
\|\tgrad \bpsi\|_{\omK}
\right )
\end{equation}
for all $\bpsi \in \BH^1_{\GN}(\Omega)$. Indeed, if $q \geq 2$, we can define
$\CJ_{hq}$ by using $\CI_{h(q-1)}$ component-wise (with boundary conditions
imposed on $\GN$ rather than on $\GD$), and if $q = 1$, a suitable operator
is designed in~\cite{hiptmair_pechstein_2017a}.

\subsubsection{Regular decomposition}
\label{section_decomposition_schoberl}

\begin{subequations}
In order to employ the quasi-interpolation operator $\CJ_{hq}$ introduced above,
it will be crucial to replace a general function $\bthe \in \BH_{\GN}(\ccurl,\Omega)$
by a function $\bpsi \in \BH^1_{\GN}(\Omega)$ with the same curl.
Specifically, we will employ \cite[Theorem 1]{schoberl_2008a}, stating
that there exists an operator
$S_h: \BH_{\GN}(\ccurl,\Omega) \to \ND_1(\CT_h) \cap \BH_{\GN}(\ccurl,\Omega)$
such that for all $\bthe \in \BH_{\GN}(\ccurl,\Omega)$,
\begin{equation}
\label{eq_decomposition_schoberl}
\bthe-S_h\bthe = \grad q + \bpsi
\end{equation}
with $q \in H^1_{\GN}(\Omega)$, $\bpsi \in \BH^1_{\GN}(\Omega)$ and
\begin{equation}
\label{eq_stability_schoberl}
h_K^{-1} \|\bpsi\|_K + \|\tgrad \bpsi\|_K \leq C(\kappa_{\tomK}) \|\curl \bthe\|_{\tomK}.
\end{equation}
\end{subequations}

\subsubsection{Inverse inequalities}

When dealing with the IPDG scheme, we will need the following results to show that
our new error estimator is smaller than the standard one. For $K \in \CT_h$ and $F \in \CF_h$,
the estimates
\begin{subequations}
\begin{equation}
\label{eq_inverse_K}
\|\grad v_q\|_K \leq C(\kappa_K) \frac{q^2}{h_K}\|v_q\|_K \qquad \forall v_q \in \CP_q(K)
\end{equation}
and
\begin{equation}
\label{eq_inverse_F}
\|\grad_F w_q\|_F \leq C(\kappa_{\omF}) \frac{q^2}{h_F}\|w_q\|_F
\qquad
\forall w_q \in \CP_q(F)
\end{equation}
are established for all $q \geq 0$ in~\cite[Theorem 4.76]{schwab_1998a}.
By combining~\eqref{eq_inverse_K} with a standard multiplicative
trace inequality, for all $K \in \CT_h$ and $F \subset \CF(K)$, we also have
\begin{equation}
\label{eq_inverse_trace}
\|v_q\|_F^2 \leq C(\kappa_K) \frac{q^2}{h_K}\|v_q\|_K^2 \qquad \forall v_q \in \CP_q(K)
\end{equation}
for all $q \geq 1$, see again~\cite[Theorem 4.76]{schwab_1998a}.
\end{subequations}

\subsection{The new error estimator}

The first term in the right-hand side of the Prager--Synge identity
in~\eqref{eq_standard_prager_synge} is usually controlled in the literature
by the construction of an equilibrated flux, or with residual terms of the form
\begin{equation*}
\frac{h_K}{p}\|\div (\AAA\GRAD)+f\|_K
+
\sqrt{\frac{h_K}{p}} \|\jmp{\AAA\GRAD} \cdot \bn_K\|_{\partial K \setminus \GD}
\end{equation*}
in each element $K \in \CT_h$, where $p$ is the polynomial degree associated with
the numerical method. This term is fine, and leads to satisfactory efficiency estimate.

Here, we rather focus on the second term in the right-hand side
of~\eqref{eq_standard_prager_synge}. In the literature, this term is either estimated by
constructing a potential as in~\cite{ainsworth_2007a,ainsworth_2008a,ern_vohralik_2015a,kim_2007a},
or by exploiting the fact that $\GRAD$ is a broken gradient of scalar potential and measuring
the jump of the potential across mesh
faces following~\cite{cangiani_dong_georgoulis_2023a,houston_schotzau_wihler_2006a,karakashian_pascal_2003a}.
The key issue is that neither approach leads to efficiency estimates that solely involve
$\grad u-\GRAD$. Here, we instead employ a residual term of the form residual terms of the form
\begin{equation*}
\frac{h_K}{p}\|\curl \GRAD\|_K
+
\sqrt{\frac{h_K}{p}} \|\jmp{\GRAD} \times \bn_K\|_{\partial K \setminus \GN}
\end{equation*}
that is still reliable, but leads to sharper lower bounds.

\subsection{Abstract reliability}
\label{section_res_abstract_rel}

We start with a reliability result under abstract assumptions on $\GRAD$.
We first focus on the second term in the Prager-Synge estimate measuring
the non-conformity of the discrete gradient, which is the new term in our
estimator.

\begin{lemma}[Abstract reliability]
\label{lemma_res_abstract_rel}
Let $\GRAD \in \BL^2(\Omega)$ and $q \geq 1$.
If $(\GRAD,\bxi_h)_\Omega = 0$ for all $\bxi_h \in \RT_q(\CT_h) \cap \BH_{\GN}(\ddiv^0,\Omega)$,
then
\begin{multline}
\label{eq_rel_res}
\max_{\substack{\bthe \in \BH_{\GN}(\ccurl,\Omega) \\ \|\curl \bthe\|_{\AAA^{-1},\Omega} = 1}}
|(\GRAD,\curl \bthe)_\Omega|^2
\leq
\\
C(\kappa_{\CT_h})^2
\sum_{K \in \CT_h}
\alpha_{\max,\ttomK}
\left \{
\left (\frac{h_K}{q}\right )^2 \|\curl \GRAD\|_K^2
+
\frac{h_K}{q} \|\jmp{\GRAD} \times \bn_F\|_{\partial K \setminus \GN}^2
\right \}.
\end{multline}
\end{lemma}

\begin{proof}
Consider an arbitrary $\bthe \in \BH_{\GN}(\ccurl,\Omega)$.
Recalling the discussion in Section \ref{section_decomposition_schoberl}
and our assumption on $\GRAD$, since
$\curl (S_h\bthe) \in \RT_1(\CT_h) \cap \BH_{\GN}(\ddiv^0,\Omega)$,
we have
\begin{equation*}
(\GRAD,\curl\bthe)_\Omega
=
(\GRAD,\curl(\bthe-S_h\bthe))_\Omega
=
(\GRAD,\curl \bpsi)_\Omega
\end{equation*}
for $\bpsi \in \BH^1_{\GN}(\Omega)$ satisfying \eqref{eq_stability_schoberl}.
Similarly, since $\CJ_{hq} \bpsi \in \ND_q(\CT_h) \cap \BH_{\GN}(\ccurl,\Omega)$,
we have $\curl(\CJ_{hq}\bpsi) \in \RT_q(\CT_h) \cap \BH_{\GN}(\ddiv^0,\Omega)$,
and due to our assumption of $\GRAD$, we can further write that
\begin{equation*}
(\GRAD,\curl \bthe)_\Omega
=
(\GRAD,\curl \widetilde \bpsi)_\Omega
\end{equation*}
with $\widetilde \bpsi = \bpsi-\CJ_{hq}\bpsi$.

We perform standard integration by parts to show that
\begin{align*}
(\GRAD,\curl \bthe)_\Omega
&=
\sum_{K \in \CT_h}
\left \{
(\curl \GRAD,\widetilde \bpsi)_K
+
(\GRAD \times \bn_K,\widetilde \bpsi)_{\partial K}
\right \}
\\
&=
\sum_{K \in \CT_h}
(\curl \GRAD,\widetilde \bpsi)_K
+
\sum_{F \in \CF_h}
(\jmp{\GRAD} \times \bn_F,\widetilde \bpsi)_{\partial K},
\\
&\leq
\sum_{K \in \CT_h}
\left \{
\|\curl \GRAD\|_K\|\widetilde \bpsi\|_K
+
\|\jmp{\GRAD} \times \bn\|_{\partial K \setminus \GN}
\|\widetilde \bpsi \times \bn\|_{\partial K}
\right \}
\end{align*}
and continue with Cauchy-Schwarz inequalities as follows
\begin{align*}
(\GRAD,\curl \bthe)_\Omega
&\leq
\sum_{K \in \CT_h}
\alpha_{\max,\ttomK}^{1/2}
\left \{
\frac{h_K}{q}\|\curl \GRAD\|_K
+
\sqrt{\frac{h_K}{q}} \|\jmp{\GRAD} \times \bn\|_{\partial K \setminus \GN}
\right \},
\\
&\qquad \qquad \times
\alpha_{\max,\ttomK}^{-1/2}
\left \{
\frac{q}{h_K}\|\widetilde \bpsi\|_K
+
\sqrt{\frac{q}{h_K}}
\|\widetilde \bpsi \times \bn\|_{\partial K}
\right \}
\\
&\leq
\sqrt{2}
\left (
\sum_{K \in \CT_h}
\alpha_{\max,\ttomK}
\left \{
\left (\frac{h_K}{q}\right )^2\|\curl \GRAD\|_K^2
+
\frac{h_K}{q} \|\jmp{\GRAD} \times \bn\|_{\partial K \setminus \GN}^2
\right \}
\right )^{1/2}
\\
&\times
\left (
\sum_{K \in \CT_h}
\alpha_{\max,\ttomK}
\left \{
\frac{q}{h_K}\|\widetilde \bpsi\|_K
+
\sqrt{\frac{q}{h_K}}
\|\widetilde \bpsi \times \bn\|_{\partial K}
\right \}^2
\right )^{1/2}.
\end{align*}

At that point, it remains to estimate the second sum.
To do so, we consider a fixed element $K \in \CT_h$.
Using \eqref{eq_quasi_interpolation} and \eqref{eq_stability_schoberl}, we have
\begin{align*}
\frac{q}{h_K}\|\widetilde \bpsi\|_K
+
\sqrt{\frac{q}{h_K}}
\|\widetilde \bpsi \times \bn\|_{\partial K}
&\leq
C(\kappa_{\CT_h})
(\ell_\Omega^{-1} \|\bpsi\|_{\omega_K}+\|\tgrad \bpsi\|_{\omega_K})
\\
&\leq
C(\kappa_{\CT_h})
\sum_{\substack{K' \in \CT_h \\ K \subset \omega_K}}
(h_K^{-1} \|\bpsi\|_{K'}+\|\tgrad \bpsi\|_{K'})
\\
&\leq
C(\kappa_{\CT_h})
\sum_{\substack{K' \in \CT_h \\ K \subset \omega_K}}
\|\curl \bthe\|_{\tomK},
\end{align*}
leading to
\begin{align*}
\alpha_{\max,\ttomK}^{-1/2}
\left \{
\frac{q}{h_K}\|\widetilde \bpsi\|_K
+
\sqrt{\frac{q}{h_K}}
\|\widetilde \bpsi \times \bn\|_{\partial K}
\right \}
&\leq
C(\kappa_{\CT_h})
\alpha_{\max,\ttomK}^{-1/2}
\sum_{\substack{K' \in \CT_h \\ K \subset \omega_K}}
\|\curl \bthe\|_{\tomK}
\\
&\leq
C(\kappa_{\CT_h})
\|\curl \bthe\|_{\AAA^{-1},\ttomK},
\end{align*}
and, due to a bounded number of overlaps, \eqref{eq_rel_res} follows by summation.
\end{proof}

Although the first two terms of the estimator are standard, for completeness,
we still analyze them here. For conforming discretizations, the treatment of
these terms can be bound, e.g.,
in~\cite{ainsworth_oden_2000a,melenk_wohlmuth_2001a,verfurth_2013a}.

\begin{lemma}[Absract reliability for the conforming part of the estimator]
\label{lemma_res_rel_div}
Let $\GRAD \in \BL^2(\Omega)$ such that $(\AAA\GRAD,\grad v_h)_\Omega = (f,v_h)_\Omega$
for all $v_h \in \CP_q(\CT_h) \cap H^1_{\GD}(\Omega)$ for some $q \geq 1$. Then, we have
\begin{multline}
\label{eq_rel_res_div}
\sup_{\substack{v \in H^1_{\GD}(\Omega) \\ \|\grad v\|_{\AAA,\Omega} = 1}}
|(f,v)_\Omega-(\AAA\GRAD,\grad v)_\Omega|^2
\leq
C(\kappa_{\CT_h})^2
\\
\sum_{K \in \CT_h}
\alpha_{\min,\ttomK}^{-1}
\left \{
\left (\frac{h_K}{q}\right )^2
\|\div(\AAA\GRAD)+f\|_K^2
+
\frac{h_K}{q}
\|\jmp{\AAA\GRAD}\cdot \bn_K\|_{\partial K \setminus \GD}^2
\right \}.
\end{multline}
\end{lemma}

\begin{proof}
We fix $v \in H^1_{\GD}(\Omega)$ and observe that by assumption, we have
\begin{equation*}
(f,v)_\Omega-(\AAA\GRAD,\grad v)_\Omega
=
(f,\widetilde v)_\Omega-(\AAA\GRAD,\grad \widetilde v)_\Omega
\end{equation*}
with $\widetilde v = v -\CI_{hq} v$, where $\CI_{hq}$ is the quasi-interpolation
operator from \eqref{eq_quasi_interpolation_scalar}. Integration by parts then gives that
\begin{align*}
(f,v)_\Omega-(\AAA\GRAD,\grad v)_\Omega
&=
\sum_{K \in \CT_h}
\left \{
(\div(\AAA\GRAD)+f,\widetilde v)_K-(\AAA\GRAD\cdot \bn_K,\widetilde v)_{\partial K}
\right \}
\\
&=
\sum_{K \in \CT_h}
(\div(\AAA\GRAD)+f,\widetilde v)_K
-
\sum_{F \in \CF_h}
(\jmp{\AAA\GRAD(u_h)}\cdot \bn_F,\widetilde v)_F
\\
&\leq
\sum_{K \in \CT_h}
\left \{
\|\div(\AAA\GRAD)+f\|_K\|\widetilde v\|_K
+
\|\jmp{\AAA\GRAD}\cdot \bn_K\|_{\partial K \setminus \GD}\|\widetilde v\|_{\partial K}
\right \},
\end{align*}
and by further employing Cauchy-Schwarz inequalities, we have
\begin{align*}
&(f,v)_\Omega-(\AAA\GRAD,\grad v)_\Omega
\\
&\leq
\sum_{K \in \CT_h}
\alpha_{\min,\omK}^{-1/2}
\left (
\frac{h_K}{q}
\|\div(\AAA\GRAD)+f\|_K
+
\sqrt{\frac{h_K}{q}}
\|\jmp{\AAA\GRAD}\cdot \bn_K\|_{\partial K \setminus \GD}
\right )
\\
&\qquad \qquad \times
\alpha_{\min,\omK}^{1/2}
\left (
\frac{q}{h_K}
\|\widetilde v\|_K
+
\sqrt{\frac{q}{h_K}}
\|\widetilde v\|_{\partial K}
\right )^2
\\
&\leq
\sqrt
{
\sum_{K \in \CT_h}
\alpha_{\min,\omK}^{-1}
\left (
\frac{h_K}{q}
\|\div(\AAA\GRAD)+f\|_K
+
\sqrt{\frac{h_K}{q}}
\|\jmp{\AAA\GRAD}\cdot \bn_K\|_{\partial K \setminus \GD}
\right )^2
}
\\
&\times
\sqrt
{
\sum_{K \in \CT_h}
\alpha_{\min,\omK}
\left (
\frac{q}{h_K}
\|\widetilde v\|_K
+
\sqrt{\frac{q}{h_K}}
\|\widetilde v\|_{\partial K}
\right )^2},
\end{align*}
and \eqref{eq_rel_res_div} follows provided we show that
\begin{equation}
\label{tmp_bound_gradv}
\sum_{K \in \CT_h}
\alpha_{\min,\omK}
\left (
\frac{q}{h_K}
\|\widetilde v\|_K
+
\sqrt{\frac{q}{h_K}}
\|\widetilde v\|_{\partial K}
\right )^2
\leq
C(\kappa_{\CT_h})\|\grad v\|_{\AAA,\Omega}^2.
\end{equation}
To that end, we employ \eqref{eq_quasi_interpolation_scalar} leading to
\begin{equation*}
\alpha_{\min,\omK}^{1/2}
\left (
\frac{q}{h_K}
\|\widetilde v\|_K
+
\sqrt{\frac{q}{h_K}} 
\|\widetilde v\|_{\partial K}
\right )
\leq
C(\kappa_{\omK}) \alpha_{\min,\omK}^{1/2} \|\grad v\|_{\omK}
\leq
C(\kappa_{\omK}) \|\grad v\|_{\AAA,\omK},
\end{equation*}
and \eqref{tmp_bound_gradv} follows by summation due to the
finite number of overlap between patches.
\end{proof}

\subsection{Abstract efficiency}
\label{section_res_abstract_eff}

The following efficiency results are proved using the standard bubble function techniques,
as introduced in Section \ref{section_bubble}. We first establish a result for the new
terms in the error estimator, and then revisit the usual proofs for the standard terms.

\begin{lemma}[Abstract efficiency]
\label{lemma_res_abstract_eff}
Let $q \geq 1$ and assume that $\GRAD \in \BCP_q(\CT_h)$. Then,
\begin{equation}
\label{eq_eff_res_K}
\frac{h_K}{q} \|\curl \GRAD\|_K \leq C(\kappa_K) q \|\grad u-\GRAD\|_K
\end{equation}
for all $K \in \CT_h$ and
\begin{equation}
\label{eq_eff_res_F}
\left (\frac{h_K}{q}\right )^{1/2} \|\jmp{\GRAD} \times \bn_F\|_F
\leq
C(\kappa_{\omF}) q^{3/2} \|\grad u-\GRAD\|_{\omF}
\end{equation}
for all $F \in \CFIh \cup \CFDh$.
\end{lemma}

\begin{proof}
Let $K \in \CT_h$. Since $\curl \GRAD \in \BCP_{q-1}(\CT_h) \subset \BCP_q(\CT_h)$,
we can apply \eqref{eq_bubble_K} with $\bv_q = \curl \GRAD$. Using the second estimate
in \eqref{eq_bubble_K}, we have
\begin{multline*}
\|b_K^{1/2}\curl\GRAD\|_K^2
=
(b_K\curl \GRAD,\curl(\GRAD-\grad u))_K
=
(\curl(b_K\curl \GRAD),\GRAD-\grad u)_K
\\
\leq
\|\curl(b_K\curl \GRAD)\|_K\|\GRAD-\grad u\|_K
\leq
C(\kappa_K) \frac{q}{h_K} \|b_K^{1/2}\curl\GRAD\|_K^2 \|\GRAD-\grad u\|_K
\end{multline*}
leading to
\begin{equation*}
\frac{h_K^2}{q^2}(b_K\curl \GRAD,\curl \GRAD)_K
\leq
C(\kappa_K)
\|\GRAD-\grad u\|_K^2.
\end{equation*}
We then show~\eqref{eq_eff_res_K} with the first estimate in~\eqref{eq_bubble_K} as follows:
\begin{equation*}
\frac{h_K^2}{q^2}\|\curl \GRAD\|_K^2
\leq
C(\kappa_K)
\frac{h_K^2}{q^2}q^2(b_K\curl \GRAD,\curl \GRAD)_K
\leq
C(\kappa_K)
q^2\|\GRAD-\grad u\|_K^2.
\end{equation*}

On the other hand, for $F \in \CFIh \cap \CFDh$, we have $\jmp{\GRAD} \times \bn_F \in \BCP_q(F)$,
and we can write
\begin{align*}
(b_F\jmp{\GRAD} \times \bn_F,\jmp{\GRAD} \times \bn_F)_F
&=
(\LE_F(\jmp{\GRAD} \times \bn_F),\jmp{\GRAD} \times \bn_F)_F
\\
&=
(\curl\LE_F(\jmp{\GRAD} \times \bn_F),\GRAD)_{\omega_F}
-
(\LE_F(\jmp{\GRAD} \times \bn_F),\grad_h\times\GRAD)_{\omega_F}
\\
&=
(\curl\LE_F(\jmp{\GRAD} \times \bn_F),\GRAD-\grad u)_{\omega_F}
-
(\LE_F(\jmp{\GRAD} \times \bn_F),\grad_h\times\GRAD)_{\omega_F},
\end{align*}
where we are allowed to plug in $\grad u$ since $u \in H^1(\omega_F)$
with $u = 0$ on $F$ when $F \in \CFDh$. For the first term, we use
\eqref{eq_stability_extension} and write that
\begin{align*}
|(\curl\LE_F(\jmp{\GRAD} \times \bn_F),\GRAD-\grad u)_{\omega_F}|
&\leq
\|\curl\LE_F(\jmp{\GRAD} \times \bn_F)\|_{\omega_F}\|\GRAD-\grad u\|_{\omega_F}
\\
&\leq
\|\curl\LE_F(\jmp{\GRAD} \times \bn_F)\|_{\omega_F}\|\GRAD-\grad u\|_{\omega_F}
\\
&\leq
C(\kappa_{\omF})\frac{q}{h_F^{1/2}} \|b_F^{1/2}\jmp{\GRAD} \times \bn_F\|_F\|\GRAD-\grad u\|_{\omega_F}.
\end{align*}
For the second term, we use \eqref{eq_eff_res_K} and \eqref{eq_stability_extension},
which gives
\begin{align*}
|(\LE_F(\jmp{\GRAD} \times \bn_F),\grad_h\times\GRAD)_{\omega_F}|
&\leq
\|\LE_F(\jmp{\GRAD} \times \bn_F)\|_{\omega_F}\|\grad_h\times\GRAD\|_{\omega_F}
\\
&\leq
C(\kappa_{\omF})\frac{h_F^{1/2}}{q} \|b_F^{1/2}\jmp{\GRAD} \times \bn_F\|_F
\frac{q^2}{h_F}\|\GRAD-\grad u\|_{\omega_F}
\\
&\leq
C(\kappa_{\omF}) \frac{q}{h_F^{1/2}} \|b_F^{1/2}\jmp{\GRAD} \times \bn_F\|_F \|\GRAD-\grad u\|_{\omega_F}.
\end{align*}
Combining the last two bounds gives that
\begin{equation*}
\|b_F^{1/2}\jmp{\GRAD}\times\bn_F\|_F
\leq
C(\kappa_{\omF}) \frac{q}{h_F^{1/2}} \|\GRAD-\grad u\|_{\omega_F},
\end{equation*}
and therefore, with~\eqref{eq_bubble_F}, we have
\begin{equation*}
\sqrt{\frac{h_F}{q}}
\|\jmp{\GRAD}\times\bn_F\|_F
\leq
C(\kappa_{\omF}) \sqrt{qh_F}
\|b_F^{1/2}\jmp{\GRAD}\times\bn_F\|_F
\leq
C(\kappa_{\omF})q^{3/2} \|\GRAD-\grad u\|_{\omega_F},
\end{equation*}
which concludes the proof of~\eqref{eq_eff_res_F}.
\end{proof}

\begin{lemma}[Abstract efficiency of the conforming part of the estimator]
\label{lemma_res_eff_div}
Let $q \geq 1$ and assume that $\GRAD \in \CP_q(\CT_h)$ and $\AAA \in \TCP_q(\CT_h)$.
We have
\begin{multline}
\label{eq_eff_res_div_K}
\alpha_{\min,K}^{-1/2}\frac{h_K}{q}\|\div(\AAA\GRAD)+f\|_K
\leq
C(\kappa_K)
\\
q \left (\sqrt{\frac{\alpha_{\max,K}}{\alpha_{\min,K}}}\|\grad u-\GRAD\|_{\AAA,K} + \alpha_{\min,K}^{-1/2}\frac{h_K}{q}\|f-\Pi_{hq} f\|_K\right )
\end{multline}
for all $K \in \CT_h$, and
\begin{multline}
\label{eq_eff_res_div_F}
\alpha_{\min,\omF}^{-1/2}\sqrt{\frac{h_F}{q}}
\|\jmp{\GRAD} \cdot \bn_F\|_F
\leq
C(\kappa_{\omF})
\\
q^{3/2} \left (
\sqrt{\frac{\alpha_{\max,\omF}}{\alpha_{\min,\omF}}} \|\grad u-\GRAD\|_{\AAA,\omF}
+
\alpha_{\min,\omF}^{-1/2}\frac{h_K}{q}\|f-\Pi_{hq} f\|_{\omF}
\right )
\end{multline}
for all $F \in \CFIh \cup \CFNh$
\end{lemma}

\begin{proof}
We first write that
\begin{equation}
\label{tmp_decomposition_oscillation}
\|\div(\AAA\GRAD) + f\|_K
=
\|\div(\AAA\GRAD) + \Pi_{hq} f\|_K
+
\|f-\Pi_{hq} f\|_K
\end{equation}
where $\Pi_{hq} f \in \CP_{q-1}(\CT_h)$ is the projection of $f$
introduced in \eqref{eq_data_oscillation}. Now, we can work with
the polynomial data $r_K \eq \div(\AAA\GRAD) + \Pi_{hq} f \in \CP_{2q-1}(K)$.
We have
\begin{multline*}
\|b_K^{1/2} r_K\|_K^2 = (b_K r_K,r_K)_K = (\Pi_{hq} f,b_K r_K)_K+(\div(\AAA\GRAD),b_Kr_K)_K
\\
=
(\Pi_{hq} f-f,b_Kr_K)_K
+
(\AAA(\grad u-\GRAD),\grad (b_Kr_K))_K
\end{multline*}
since $b_K r_K \in H^1_0(K)$. For the first term, we write
\begin{equation*}
(\Pi_{hq} f-f,b_Kr_K)_K
\leq
\|f-\Pi_{hq} f\|_K\|b_Kr_K\|_K
\leq
C(\kappa_K)\|f-\Pi_{hq} f\|_K\|b_K^{1/2}r_K\|_K
\end{equation*}
As for the second term, we have
\begin{multline*}
(\AAA(\grad u-\GRAD),\grad (b_Kr_K))_K
\leq
\|\grad u - \GRAD\|_{\AAA,K}\alpha_{\max,K}^{1/2} \|\grad(b_Kr_K)\|_K
\\
\leq
C(\kappa_K) \frac{q}{h_K} \|\grad u-\GRAD\|_{\AAA,K}\alpha_{\max,K}^{1/2}\|b_K^{1/2}r_K\|_K,
\end{multline*}
leading to
\begin{equation*}
\frac{h_K}{q}\|r_K\|_K
\leq
C(\kappa_K) q \left (
\alpha_{\max,K}^{1/2}\|\grad u-\GRAD\|_{\AAA,K}
+
\frac{h_K}{q}\|f-\Pi_{hq} f\|_K
\right ),
\end{equation*}
due to \eqref{eq_bubble_K_scalar}. Then, \eqref{eq_eff_res_div_K} simply follows from
\eqref{tmp_decomposition_oscillation}.

For \eqref{eq_eff_res_div_F}, we fix $F \in \CFIh \cup \CFNh$ and
let $r_F \eq \jmp{\AAA\GRAD} \cdot \bn_F \in \CP_q(F)$. Then,
we write that
\begin{align*}
\|b_F^{1/2}r_F\|_F^2
&=
(r_F,b_Fr_F)_F
=
(r_F,\LE_F(r_F))_F
\\
&=
(\AAA\GRAD,\grad \LE_F(r_F))_{\omF}
+
(\grad_h \cdot(\AAA\GRAD),\LE_F(r_F))_{\omF}
\\
&=
(\AAA(\GRAD-\grad u),\grad \LE_F(r_F))_{\omF}
+
(\grad_h \cdot(\AAA\GRAD)+f,\LE_F(r_F))_{\omF}
\\
&\leq
\|\GRAD-\grad u\|_{\AAA,\omF}\alpha_{\max,\omF}^{1/2}\|\grad \LE_F(r_F)\|_{\omF}
+
\|\grad_h \cdot(\AAA(\GRAD)+f\|_{\omF}\|\LE(r_F)\|_{\omF}
\end{align*}
and it follows from \eqref{eq_stability_extension_scalar} that
\begin{equation*}
\|b_F^{1/2} r_F\|_F
\leq
C(\kappa_{\omF})
\left (
\alpha_{\max,\omF}^{1/2} q h_F^{-1/2} \|\GRAD -\grad u\|_{\AAA,\omF}
+
q^{-1} h_F^{1/2} \|\grad_h \cdot(\AAA\GRAD)+f\|_{\omF}
\right )
\end{equation*}
and
\begin{equation*}
\sqrt{\frac{h_F}{p}}
\|r_F\|_F
\leq
C(\kappa_{\omF})
\left (
\alpha_{\max,\omF}^{1/2} q^{3/2} \|\GRAD-\grad u\|_{\AAA,\omF}
+
q^{1/2}\frac{h_F}{q} \|\grad_h \cdot(\AAA\GRAD_h)+f\|_{\omF}
\right )
\end{equation*}
due to \eqref{eq_bubble_F_scalar}. Then, \eqref{eq_eff_res_div_F}
follows from \eqref{eq_eff_res_div_K}.
\end{proof}

\begin{remark}[Piecewise polynomial coefficient]
It is interesting to remark that we do not require any piecewise polynomial
assumption on $\AAA$ in Lemma~\ref{lemma_res_abstract_eff}, whereas such an
assumption is necessary in Lemma~\ref{lemma_res_eff_div}. There, we have assumed
for simplicity that $\AAA \in \TCP_q(\CT_h)$, but we could more generally assume
that $\AAA \in \TCP_r(\CT_h)$ for any $r \geq 1$ if we replace the first occurences
of $q$ by $p+r$ in the right-hand sides of~\eqref{eq_eff_res_div_K} and~\eqref{eq_eff_res_div_F}.
\end{remark}

\subsection{An alternative error estimator for the IPDG scheme}

We are now ready to propose a new error estimate for the IPDG scheme
introduced in Section~\ref{section_ipdg_scheme}.
Specifically, for each $K \in \CT_h$, we introduce the estimator
\begin{align*}
\eta_K^2
&\eq
\frac{1}{\alpha_{\min,\ttomK}}\left (\frac{h_K}{p}\right )^2 \|\div(\AAA\GRAD_h(u_h))+f\|_K^2
+
\frac{1}{\alpha_{\min,\ttomK}} \frac{h_K}{p} \|\jmp{\AAA \GRAD_h(u_h)} \cdot \bn\|_{\partial K \setminus \GD}^2
\\
&+
\alpha_{\max,\ttomK} \left (\frac{h_K}{p}\right )^2 \|\curl \GRAD_h(u_h)\|_{K}^2
+
\alpha_{\max,\ttomK} \frac{h_K}{p} \|\jmp{\GRAD_h(u_h)} \times \bn\|_{\partial K \setminus \GN}^2.
\end{align*}
We have the following estimates:

\begin{theorem}[Residual-based estimator for IPDG]
\label{theorem_res_ipdg}
We have
\begin{equation}
\label{eq_rel_res_ipdg}
\|\GRAD_h(u-u_h)\|_{\AAA,\Omega}^2 \leq C(\kappa_{\CT_h}) (1+\gamma_h)\sum_{K \in \CT_h} \eta_K^2.
\end{equation}
Furthermore, if $\AAA \in \TCP_p(\CT_h)$, then
\begin{equation}
\label{eq_eff_res_ipdg}
\eta_K
\leq
C(\kappa_{\omKF})
p^{3/2} \left (
\sqrt{\frac{\alpha_{\max,\ttomK}}{\alpha_{\min,\ttomK}}} \|\GRAD_h(u-u_h)\|_{\AAA,\omKF}
+
\alpha_{\min,\omKF}^{-1/2}\frac{h_K}{q}\|f-\Pi_{hp} f\|_{\omKF}
\right )
\end{equation}
for all $K \in \CT_h$.
\end{theorem}

\begin{proof}
For the upper bound in~\eqref{eq_rel_res_ipdg}, we invoke the generalized Prager--Synge
inequality in~\eqref{eq_generalized_prager_synge}, which is feasible since
\begin{equation}
\label{tmp_assumption_discrete_gradient}
(\GRAD_h(u_h),\bxi_h)_\Omega
=
(u_h,\div \bxi_h)_\Omega
=
0
\end{equation}
for all $\bxi_h \in \RT_p(\CT_h) \cap \BH_{\GN}(\ddiv^0,\Omega)$ due to
\eqref{eq_stokes_discrete_gradient}. We then bound the two terms using Lemmas
\ref{lemma_res_abstract_rel} and \ref{lemma_res_rel_div} with $\GRAD = \GRAD_h(u_h)$
and $q=p$. The assumptions of Lemma \ref{lemma_res_abstract_rel} are satisfied due
to \eqref{tmp_assumption_discrete_gradient} and the ones of
Lemma \ref{lemma_res_rel_div} due to \eqref{eq_ipdg_galerkin_orthogonality}.

As for~\eqref{eq_eff_res_ipdg}, we simply apply Lemmas \ref{lemma_res_abstract_eff}
and \ref{lemma_res_eff_div}, again with $\GRAD = \GRAD_h(u_h)$ and $q=p$.
\end{proof}

The first two terms in the definition of the estimator $\eta_K$ are standard:
They are essentially already present for conforming finite element discretizations, see
e.g.~\cite{ainsworth_oden_2000a,melenk_wohlmuth_2001a,verfurth_2013a}.
The novelty of the proposed estimator is the way the non-conformity of the scheme is handled.


The standard error estimator for the IPDG scheme typically reads
\begin{align}
\label{eq_standard_IPDG_estimator}
\widetilde \eta_K^2
&\eq
\frac{1}{\alpha_{\min,\ttomK}}\left (\frac{h_K}{p}\right )^2 \|\div(\AAA\grad u_h)+f\|_K^2
+
\frac{1}{\alpha_{\min,\ttomK}} \frac{h_K}{p} \|\jmp{\AAA \grad u_h} \cdot \bn\|_{\partial K \setminus \GD}^2
\\
\nonumber
&+
\alpha_{\max,\ttomK} \frac{p^2}{h_K} \|\jmp{u_h}\|_{\partial K \setminus \GN}^2
\end{align}
for all $K \in \CT_h$.
The next result shows that the proposed estimator
is sharper than the standard one, in the sense that it is smaller while
remaining efficient (in a norm with less terms).

It follows from the discrete trace inequality in~\eqref{eq_inverse_trace} that
\begin{equation}
\label{eq_bound_lifting}
\|\LIFT_h(u_h)\|_{K}^2
\leq
C(\kappa_K) \frac{p^2}{h_K} \|\jmp{u_h}\|_{\partial K \setminus \GN}^2.
\end{equation}
As a result, we can first note that
\begin{equation*}
\|\grad u_h-\GRAD_h(u_h)\|_{\AAA,K}^2
=
\|\LIFT_h(u_h)\|_{\AAA,K}^2
\leq
C(\kappa_K)
\alpha_{\max,K} \frac{p^2}{h_K} \|\jmp{u_h}\|_{\partial K \setminus \GN}^2
\end{equation*}
and that therefore using the broken gradient $\grad_h u_h$ or the discrete gradient
$\GRAD_h(u_h)$ in~\eqref{eq_standard_IPDG_estimator} is not important as soon as the
jumps of $u_h$ are included in the estimator and the error measure.

\begin{theorem}[Comparison with the standard estimator]
\begin{subequations}
\label{eq_res_comparison}
For $K \in \CT_h$, we have
\begin{equation}
\label{eq_res_comparison_K}
\frac{h_K^2}{p^2}\|\curl \GRAD_h(u_h)\|_K^2
\leq
C(\kappa_K) p^2
\left (
\frac{p^2}{h_K}\|\jmp{u_h}\|_{\partial K \setminus \GN}^2
\right )
\end{equation}
and
\begin{equation}
\label{eq_res_comparison_F}
\frac{h_K}{p}\|\jmp{\GRAD_h(u_h)} \times \bn_F\|_{\partial K}^2
\leq
C(\kappa_{\omKF}) p^2
\left (
\sum_{\substack{K' \in \CT_h \\ K' \subset \omKF}}
\frac{p^2}{h_{K'}}\|\jmp{u_h}\|_{\partial K' \setminus \GN}^2
\right ).
\end{equation}
\end{subequations}
In particular, we have
\begin{equation}
\label{eq_res_div_comparison_K}
\eta_K^2 \leq C(\kappa_{\omKF}) p^2
\sum_{\substack{K' \in \CT_h \\ K' \subset \omKF}} \widetilde \eta_{K'}^2
\end{equation}
for all $K \in \CT_h$.
\end{theorem}

\begin{proof}
Let $K \in \CT_h$. To establish~\eqref{eq_res_comparison_K},
we first invoke the inverse inequality in~\eqref{eq_inverse_K}, leading to
\begin{equation*}
\frac{h_K^2}{p^2}\|\curl \GRAD_h(u_h)\|_K^2
=
\frac{h_K^2}{p^2}\|\curl \LIFT_h(\jmp{u_h})\|_K^2
\leq
C(\kappa_K) p^2\|\LIFT_h(\jmp{u_h})\|_K^2,
\end{equation*}
and the result follows from \eqref{eq_bound_lifting}. To prove~\eqref{eq_res_comparison_F},
we first write that
\begin{align*}
\|\jmp{\GRAD_h(u_h)} \times \bn_F\|_F
&\leq
\|\jmp{\grad_h u_h} \times \bn_F\|_F
+
\|\jmp{\LIFT(u_h)} \times \bn_F\|_F
\\
&=
\|\grad_F \jmp{u_h}\|_F
+
\|\jmp{\LIFT(u_h)} \times \bn_F\|_F.
\end{align*}
For the first term, we simply employ the inverse inequality in~\eqref{eq_inverse_F}.
For the second term, if $F = \partial K_- \cap \partial K_+ \in \CFIh$ is an interior face,
we write
\begin{equation*}
\|\jmp{\LIFT(u_h)} \times \bn_F\|_F
\leq
\|\LIFT(u_h)|_{K_-} \times \bn_F\|_F
+
\|\LIFT(u_h)|_{K_+} \times \bn_F\|_F
\end{equation*}
and simply
\begin{equation*}
\|\jmp{\LIFT(u_h)} \times \bn_F\|_F
\leq
\|\LIFT(u_h)|_{K} \times \bn_F\|_F
\end{equation*}
for $F \in \CFDh$, where $K \in \CT_h$ is the only element such that $F \subset \partial K$.
We will only detail the second case. Using the trace inequality in~\eqref{eq_inverse_trace},
we have
\begin{equation*}
\|\LIFT(u_h) \times \bn_F\|_F
\leq
\|\LIFT(u_h)\|_F
\leq
C(\kappa_K)
\frac{p}{\sqrt{h_K}}
\|\LIFT(u_h)\|_K
\leq
C(\kappa_K) \frac{p^2}{h_K} \|\jmp{u_h}\|_{\partial K \setminus \GN},
\end{equation*}
where we employ~\eqref{eq_bound_lifting} in the last inequality.
\end{proof}

\subsection{An alternative error estimator for mixed FEM discretizations}

We next consider the mixed FEM discretization presented in Section~\ref{section_mixed_fem}.
Here, our error estimator reads
\begin{equation*}
\eta_K^2
\eq
\alpha_{\max,\ttomK}
\left \{
\left (
\frac{h_K}{p}
\right )^2 \|\curl (\AAA^{-1}\bsig_h)\|_K^2
+
\frac{h_K}{p}
\|\jmp{\AAA^{-1}\bsig_h} \times \bn\|_{\partial K \setminus \GN}^2
\right \}
\end{equation*}
for all $K \in \CT_h$.

\begin{theorem}[Residual-based estimator for mixed FEM]
\label{theorem_res_mixed_fem}
We have
\begin{equation}
\label{eq_rel_res_mixed_fem}
\|\AAA\grad u+\bsig_h\|_{\AAA^{-1},\Omega}^2
\leq
C(\kappa_{\CT_h})(1+\gamma_h)
\sum_{K \in \CT_h}
\left \{
\eta_K^2
+
\left (\frac{h_K}{p}\right )^2 \|f-\Pi_{hp}f\|_K^2
\right \}
\end{equation}
and, if $\AAA \in \TCP_p(\CT_h)$, then
\begin{equation}
\label{eq_eff_res_mixed_fem}
\eta_K
\leq
C(\kappa_{\omKF}) \sqrt{\frac{\alpha_{\max,\ttomK}}{\alpha_{\min,\ttomK}}}
p^{3/2} \|\AAA\grad u+\bsig_h\|_{\AAA^{-1},\omKF}
\end{equation}
for all $K \in \CT_h$.
\end{theorem}

\begin{proof}
We first show~\eqref{eq_rel_res_mixed_fem}.
Here, we utilize the gerenalized Prager--Synge
estimate~\eqref{eq_generalized_prager_synge} with $\GRAD = -\AAA^{-1}\bsig$,
and we estimate the second term using Lemma~\ref{lemma_res_abstract_rel}.
Notice that the assumptions of both results are satisfied due
to~\eqref{eq_orthogonality_mixed_fem}. For the first term, for $v \in H^1_{\GD}(\Omega)$, we write
\begin{equation*}
(f,v)_\Omega-(\AAA\GRAD,\grad v)_\Omega
=
(f,v)_\Omega-(\div \bsig_h,v)_\Omega
=
(f-\Pi_{hp}f,v)_\Omega
=
(f-\Pi_{hp}f,v-\Pi_{hp} v)_\Omega,
\end{equation*}
and the conclusion follows from \eqref{eq_data_oscillation_hp}, since
\begin{align*}
(f-\Pi_{hp}f,v-\Pi_{hp} v)_\Omega
&\leq
C(\kappa_{\CT_h})
\sum_{K \in \CT_h}
\frac{h_K}{p}\|f-\Pi_{hp}f\|_K
\|\grad v\|_K
\\
&\leq
C(\kappa_{\CT_h})
\sum_{K \in \CT_h}
\alpha_{\min,K}^{-1/2}
\frac{h_K}{p}\|f-\Pi_{hp}f\|_K
\|\grad v\|_{\AAA,K}.
\end{align*}

For~\eqref{eq_eff_res_mixed_fem}, since
$\AAA \in \TCP_p(\CT_h)$ by assumption here,
we can apply Lemma~\ref{lemma_res_abstract_eff}
with $\GRAD = \AAA^{-1}\bsig_h$, leading to result
\begin{align*}
\eta_K
&\leq
C(\kappa_{\omKF}) \alpha_{\max,\ttomK}^{1/2}
p^{3/2} \|\grad u-\AAA^{-1}\bsig_h\|^2_{\omKF}
\\
&\leq
C(\kappa_{\omKF}) \sqrt{\frac{\alpha_{\max,\ttomK}}{\alpha_{\min,\ttomK}}}
p^{3/2} \|\grad u-\AAA^{-1}\bsig_h\|^2_{\AAA,\omKF}
\\
&\leq
C(\kappa_{\omKF}) \sqrt{\frac{\alpha_{\max,\ttomK}}{\alpha_{\min,\ttomK}}}
p^{3/2} \|\AAA\grad u-\bsig_h\|^2_{\AAA^{-1},\omKF}.
\end{align*}
\end{proof}

\section{Fully-equilibrated estimators}
\label{section_equilibrated}

We now propose to exploit the alternative Prager--Synge estimate
in~\eqref{eq_generalized_prager_synge_equilibrated} to derive a
flux-equilibrated error estimator. Crucially, in contrast to the
standard theory in~\cite{ainsworth_2007a,ainsworth_2008a,ern_vohralik_2015a,kim_2007a},
we will reconstruct a curl-free approximation of the flux variable, rather than a gradient field.

\subsection{Tools}

We will need some tools developed for curl-constrained equilibration
developed in the context of the discretization of curl-curl problems
in~\cite{chaumontfrelet_vohralik_2023a}.

\subsubsection{Vertex partition of unity}

For all vertices $\ba \in \CV_h$, there exists a unique ``hat'' function
$\pa \in \CP_1(\CT_h) \cap H^1(\Omega)$ such that $\pa(\bb) = \delta_{\ba,\bb}$
for all vertices $\bb \in \CV_h$. We have $\pa \in H^1(\Omega)$ and
\begin{equation}
\label{eq_pa}
\|\pa\|_{L^\infty(\Omega)}
\leq
1,
\qquad
\|\grad \pa\|_{\BL^\infty(\Omega)}
\leq
C(\kappa_{\CTa})
\ha.
\end{equation}
Each hat function $\pa$ is supported in $\oma$, and they form a partition of unity, i.e.
\begin{equation}
\label{eq_pa_pu}
\sum_{\ba \in \CV_h} \pa = 1.
\end{equation}

\subsubsection{Local function spaces and inequalities}

Consider a vertex $\ba \in \CV_h$. If $\ba \notin \overline{\GN}$,
we let $\gamma_{\rm N} \eq \emptyset$ and $\gamma_{\rm N}^{\rm c} \eq \partial \oma$.
On the other, when $\ba \in \overline{\GN}$, we denote by $\gamma_{\rm N}$ the set
covered by the faces sharing the vertex $\ba$ and lying in $\overline{\GN}$,
whereas $\gamma_{\rm N}^{\rm c}$ consists of the remainder of $\partial \oma$. We then let
$\BH_0(\ccurl,\oma) \eq \BH_{\gamma_{\rm N}^{\rm c}}(\ccurl,\oma)$
and
$\BH_\star(\ccurl,\oma) \eq \BH_{\gamma_{\rm N}}(\ccurl,\oma)$,
as well as
$\BH_0(\ddiv^0,\oma) \eq \BH_{\gamma_{\rm N}^{\rm c}}(\ddiv^0,\oma)$.
Finally, we define $H^1_\star(\oma)$  by setting
$H^1_\star(\oma) \eq H^1_{\gamma_{\rm N}}(\oma)$
when $\gamma_{\rm N} \neq \emptyset$, and as the subspace of $H^1(\oma)$
of functions with vanishing mean-value otherwise.

We have the following Poincar\'e-like inequality
\begin{equation}
\label{eq_poincare_curl}
\|\bv\|_{\oma} \leq C(\kappa_{\CTa}) \ha \|\curl \bv\|_{\oma}
\end{equation}
for all $\bv \in \BH_\star(\ccurl,\oma) \cap \BH_0(\ddiv^0,\oma)$.
The fact that the constant only depends on $\kappa_{\CTa}$ can
be shown following~\cite[Appendix A]{chaumontfrelet_ern_vohralik_2022a}.

\subsection{Curl-free field reconstruction}

In this section, given $\GRAD \in \BL^2(\Omega)$ with $(\GRAD,\bxi_h)_\Omega = 0$
for all $\bxi_h \in \RT_1(\CT_h) \cap \BH_{\GN}(\ddiv^0,\Omega)$, we follow the equilibration
philosophy and propose a strategy to build a field
$\bphi_h \in \ND_p(\CT_h) \cap \BH_{\GD}(\ccurl^0,\Omega)$
in order to bound the error as follows:
\begin{equation*}
\min_{s \in H^1_{\GD}(\Omega)}
\|\GRAD-\grad s\|_{\AAA,\Omega}
=
\min_{\bphi \in \BH_{\GD}(\ccurl^0,\Omega)}
\|\GRAD-\bphi\|_{\AAA,\Omega}
\leq
\|\GRAD-\bphi_h\|_{\AAA,\Omega}.
\end{equation*}

\subsubsection{Main ideas}

We follow the construction introduced in~\cite{chaumontfrelet_vohralik_2023a},
where the idea is to construct $\bphi_h$ from local
contributions $\bphi_h^{\ba} \in \BH_0(\ccurl,\oma)$
meant to mimic $\pa\GRAD$.
The product rule
\begin{equation*}
\curl(\pa\GRAD) = \grad \pa \times \GRAD+\pa\curl\GRAD
\end{equation*}
motivates that each $\bphi_h^{\ba}$ should be constructed
in such way that $\curl \bphi_h^{\ba} = \grad \pa \times \GRAD$.
Unfortunately, this is not possible since in general
$\grad \pa \times \GRAD \notin \BH_0(\ddiv^0,\oma)$
for the methods considered in this work. We therefore
introduce a suitable local curl-constrained
$\cor^{\ba}_h \sim \grad \pa \times \GRAD$ such that
$\cor^{\ba}_h \in \BH_0(\ddiv^0,\oma)$.
Following~\cite{chaumontfrelet_vohralik_2023a}, this construction is done locally
in two steps, whereby $\cor^{\ba}_h = \widehat \cor^{\ba}_h - \widetilde \cor^{\ba}_h$.

\subsubsection{Reconstruction}

Our reconstruction is rigorously constructed as follows:

\begin{theorem}[Curl-free reconstruction]
\label{theorem_curl_free_reconstruction}
Assume that $\GRAD \in \BCP_q(\CT_h)$, $q \geq 1$, satisfies $(\GRAD,\bxi_h)_\Omega = 0$
for all $\bxi_h \in \RT_1(\CT_h) \cap \BH_{\GN}(\ddiv^0,\Omega)$. Then,
the following holds true.

First (i), for all $\ba \in \CV_h$, the over-constrained local problems
\begin{equation}
\label{eq_hat_cor_ba}
\widehat \cor^{\ba}_h
\eq
\arg \min_{\substack{
\bw_h \in \RT_{q+1}(\CTa) \cap \BH_0(\ddiv,\oma)
\\
(\bw_h,\br_0)_{\oma} = (\grad \pa \times \GRAD,\br_0)_{\oma}
\;
\forall \br_0 \in \BCP_0(\CTa)
\\
\div \bw_h = 0
}}
\|\grad \pa \times \GRAD-\bw_h\|_{\AAA,\oma}
\end{equation}
are well-posed, and we set
\begin{equation}
\label{eq_hat_cor}
\widehat \cor_h \eq \sum_{\ba \in \CT_h} \widehat \cor_h^{\ba}.
\end{equation}
Besides (ii), for all $\ba \in \CV_h$ we introduce $\widetilde \cor^{\ba}_h$
by solving the well-posed problem
\begin{equation}
\label{eq_tilde_cor_ba}
\widetilde \cor^{\ba}_h|_K
\eq
\arg \min_{\substack
{
\bv_h \in \RT_{q+2}(K)
\\
\div \bv_h = 0
\\
\bv_h \cdot \bn_K = \pa \widehat \cor_h \cdot \bn_K
}}
\|\bv_h-\pa\widehat \cor_h\|_K
\end{equation}
for all $K \in \CTa$, and we set
\begin{equation}
\label{eq_cor_ba}
\cor^{\ba}_h \eq \widehat \cor^{\ba}_h-\widetilde \cor^{\ba}_h.
\end{equation}
Then (iii), for all $\ba \in \CV_h$, the local problem
\begin{equation}
\label{eq_phi_ba}
\bphi^{\ba}_h
\eq
\arg \min_{\substack{
\bv_h \in \ND_{q+2}(\CT_h) \cap \BH_0(\ccurl,\oma)
\\
\curl \bv_h = \cor_h^{\ba}
}}
\|\pa\GRAD-\bv_h\|_{\AAA,\oma}
\end{equation}
is well-posed, and we have
\begin{equation}
\label{eq_phi}
\bphi_h \eq \sum_{\ba \in \CV_h} \bphi_h^{\ba} \in \ND_{q+2}(\CT_h) \cap \BH_{\GD}(\ccurl^0,\Omega).
\end{equation}
\end{theorem}

\begin{remark}[Lowest-order case]
For simplicity, the case $q=0$ has been excluded from Theorem~\ref{theorem_curl_free_reconstruction}.
However, this case can also be handled with minor modificiations following~\cite{chaumontfrelet_vohralik_2023a}.
\end{remark}

\begin{proof}
We start with \eqref{eq_hat_cor_ba}, which is treated
in~\cite[Appendix A]{chaumontfrelet_vohralik_2023a} with
$\btau_h^{\ba} = \grad \pa \times \GRAD$ and $g^{\ba} = 0$.
We first need to check~\cite[Assumption A.1]{chaumontfrelet_vohralik_2023a},
and the only non-trivial point is that
\begin{equation*}
(\btau^{\ba},\grad q_h)_{\oma} + (g^{\ba},q_h)_{\oma}
=
(\grad \pa \times \GRAD,\grad q_h)_{\oma}
=
(\GRAD,\grad q_h \times \grad \pa)_{\Omega}
=
0
\end{equation*}
for all $q_h \in \CP_1(\CTa) \cap H^1_\star(\oma)$, due to our assumptions since
$\grad p_h \times \grad \pa \in \RT_1(\CT_h) \cap \BH_{\GN}(\ddiv^0,\Omega)$.
Then, the well-posedness of \eqref{eq_hat_cor_ba} follows
from~\cite[Theorem A.2]{chaumontfrelet_vohralik_2023a}.

We then turn to \eqref{eq_tilde_cor_ba}, which is analyzed
in~\cite[Appendix B]{chaumontfrelet_vohralik_2023a}. The well-posedness
of \eqref{eq_tilde_cor_ba} follows from the observation that for all
$K \in \CT_h$
\begin{equation*}
(\widehat \cor_h,1)_{\partial K}
=
\sum_{\ba \in \CV(K)}
(\widehat \cor_h^{\ba},1)_{\partial K}
=
\sum_{\ba \in \CV(K)}
(\div \widehat \cor_h^{\ba},1)_{K}
=
\sum_{\ba \in \CV(K)}
(\grad \pa \times \GRAD,1)_{K}
=
(\grad (1) \times \GRAD,1)_{K}
=
0
\end{equation*}
due to \eqref{eq_pa_pu}.

Since $\cor_h^{\ba} \in \RT_{q+2}(\CTa) \cap \BH_0(\ddiv,\oma)$ by construction,
it is clear that the local problem defining $\bphi_h^{\ba}$ in~\eqref{eq_phi_ba}
is well-posed for each $\ba \in \CV_h$, and the fact that $\bphi \in \BH_{\GD}(\ccurl^0,\Omega)$
follows from~\cite[Theorem B.1]{chaumontfrelet_vohralik_2023a}.
\end{proof}

\subsection{A posteriori error estimates}

We are now ready to show that the proposed reconstruction
may be used to produce sharp upper and lower bounds on the
error. The upper bounds immediately follows from Theorems~\ref{theorem_curl_identity}
and~\ref{theorem_curl_free_reconstruction}

\begin{corollary}[Abstract reliability]
\label{corollary_flux_abstract_rel}
Let $\GRAD \in \BCP_q(\CT_h)$, $q\geq1$, with $(\GRAD,\bxi_h)_\Omega = 0$
for all $\bxi_h \in \RT_1(\CT_h) \cap \BH_{\GN}(\ddiv^0,\Omega)$
and let $\bphi_h \in \BH_{\GD}(\ccurl^0,\Omega)$ be constructed from~\eqref{eq_phi}
in Theorem~\ref{theorem_curl_free_reconstruction}. Then, we have
\begin{equation*}
\min_{s \in H^1_{\GD}(\Omega)} \|\GRAD-\grad s\|_{\AAA,\Omega}
\leq
\|\GRAD-\bphi_h\|_{\AAA,\Omega}.
\end{equation*}
\end{corollary}

We next establish local lower bounds. Crucially, there is no dependence
on the polynomial degree $q$ in those bound.

\begin{theorem}[Abstract efficiency]
\label{theorem_flux_abstract_eff}
Assume that $\GRAD \in \BCP_q(\CT_h)$, $q\geq1$, is such that $(\GRAD,\bxi_h)_\Omega = 0$
for all $\bxi_h \in \RT_1(\CT_h) \cap \BH_{\GN}(\ddiv^0,\Omega)$.
For $\ba \in \CV_h$, let $\bphi_h^{\ba}$ be defined from~\eqref{eq_phi_ba},
and $\bphi_h$ be constructed with~\eqref{eq_phi}.
Then, we have
\begin{equation}
\label{eq_flux_eff_vertex}
\|\pa\GRAD-\bphi_h^{\ba}\|_{\AAA,\oma}
\leq
C(\kappa_{\tCTa})
\left (\frac{\alpha_{\max,\toma}}{\alpha_{\min,\toma}}\right )^{3/2}
\|\grad u-\GRAD\|_{\AAA,\toma}
\end{equation}
for all $\ba \in \CT_h$. In addition, for all $K \in \CT_h$, the estimate
\begin{equation}
\label{eq_flux_eff_elem}
\|\GRAD-\bphi_h\|_{\AAA,K}
\leq
C(\kappa_{\tomK})
\left (\frac{\alpha_{\max,\tomK}}{\alpha_{\min,\tomK}}\right )^{3/2}
\|\grad u-\GRAD\|_{\AAA,\tomK}
\end{equation}
holds true.
\end{theorem}

\begin{proof}
We first invoke~\cite[Theorem A.2]{chaumontfrelet_vohralik_2023a}, which states that
\begin{equation*}
\|\pa\GRAD-\bphi_h^{\ba}\|_{\AAA,\oma}
\leq
C(\kappa_{\CTa})
\sqrt{\frac{\alpha_{\max,\oma}}{\alpha_{\min,\oma}}}
\min_{\substack{\bphi^{\ba} \in \BH_0(\ccurl,\oma) \\ \curl \bphi^{\ba} = \cor_h^{\ba}}}
\|\pa\GRAD-\bphi^{\ba}\|_{\AAA,\oma}.
\end{equation*}
The result in~\cite{chaumontfrelet_vohralik_2023a} is given in unweighted norms,
but the estimate above follows from norm equivalence.
The Euler-Lagrange equation for the minimizer in the right-hand side consists in
finding $\bphi^{\ba} \in \BH_0(\ccurl,\oma)$ and $\brho^{\ba} \in \BH_0(\ddiv^0,\oma)$
such that
\begin{equation*}
\left \{
\begin{array}{rcll}
(\AAA \bphi^{\ba},\bv)_{\oma} + (\brho^{\ba},\curl \bv)_{\oma} &=& (\AAA\pa\GRAD,\bv)_{\oma}
\\
(\curl \bphi^{\ba},\bw)_{\oma} &=& (\cor_h^{\ba},\bw)_{\oma}
\end{array}
\right .
\end{equation*}
for all $\bv \in \BH_0(\ccurl,\oma)$ and $\bw \in \BH_0(\ddiv^0,\oma)$.
The first equation guarantees that $\brho^{\ba} \in \BH_\star(\ccurl,\oma)$ with
$\curl \brho^{\ba} = \AAA(\bphi^{\ba}-\pa\GRAD)$. We then write that
\begin{align*}
\|\curl \brho^{\ba}\|_{\AAA^{-1},\oma}^2
&=
(\bphi^{\ba}-\pa\GRAD,\curl \brho^{\ba})_{\oma}
\\
&=
-(\pa(\grad u-\GRAD),\curl \brho^{\ba})_{\oma}
+
(\bphi^{\ba}-\pa\grad u,\curl \brho^{\ba})_{\oma}
\\
&=
-(\pa(\grad u-\GRAD),\curl \brho^{\ba})_{\oma}
+
(\curl(\bphi^{\ba}-\pa\grad u),\brho^{\ba})_{\oma}
\\
&=
-(\AAA^{-1}\AAA\pa(\grad u-\GRAD),\curl \brho^{\ba})_{\oma}
+
(\AAA^{-1}\AAA(\cor_h^{\ba}-\grad\pa\times\grad u),\brho^{\ba})_{\oma}
\\
&\leq
\|\pa(\grad u-\GRAD)\|_{\AAA,\oma}
\|\curl \brho^{\ba}\|_{\AAA^{-1},\oma}
+
\|\cor_h^{\ba}-\grad\pa\times\grad u\|_{\AAA,\oma}
\|\brho^{\ba}\|_{\AAA^{-1},\oma}
\\
&\leq
\left \{
\|\grad u-\GRAD\|_{\AAA,\oma}
+
\sqrt{\frac{\alpha_{\max,\oma}}{\alpha_{\min,\oma}}}\ha
\|\cor_h^{\ba}-\grad\pa\times\grad u\|_{\AAA,\oma}
\right \}
\|\curl \brho^{\ba}\|_{\AAA^{-1},\oma},
\end{align*}
where we employed~\eqref{eq_poincare_curl}. This leads to
\begin{equation*}
\|\bphi^{\ba}-\pa\GRAD\|_{\AAA,\oma}
\leq
\|\grad u-\GRAD\|_{\AAA,\oma}
+
C(\kappa_{\CTa})
\sqrt{\frac{\alpha_{\max,\oma}}{\alpha_{\min,\oma}}}\ha
\|\cor_h^{\ba}-\grad \pa \times \grad u\|_{\AAA,\oma}
\end{equation*}
and
\begin{align*}
\ha \|\cor_h^{\ba}-\grad \pa \times \grad u\|_{\AAA,\oma}
&\leq
\ha \|\grad \pa \times (\grad u-\GRAD)\|_{\AAA,\oma}
+
\ha \|\cor_h^{\ba}-\grad \pa \times \GRAD\|_{\AAA,\oma}
\\
&\leq
C(\kappa_{\CTa}) \|\grad u-\GRAD\|_{\AAA,\oma}
+
\ha \|\cor_h^{\ba}-\grad \pa \times \GRAD\|_{\AAA,\oma},
\\
&\leq
C(\kappa_{\CTa}) \|\grad u-\GRAD\|_{\AAA,\oma}
+
\ha \|\widehat \cor_h^{\ba}-\grad \pa \times \GRAD\|_{\AAA,\oma}
+
\ha \|\widetilde \cor_h^{\ba}\|_{\AAA,\oma}
\end{align*}
due to \eqref{eq_pa} and the definition of $\cor_h^{\ba}$.
To prove~\eqref{eq_flux_eff_vertex}, we estimate the two terms
in the right-hand side.

We estimate the first term with
\begin{align}
\label{tmp_estimate_cor}
\ha \|\widehat \cor_h^{\ba}-\grad \pa \times \GRAD\|_{\AAA,\oma}
&\leq
C(\kappa_{\CTa})
\sqrt{\frac{\alpha_{\max,\oma}}{\alpha_{\min,\oma}}}
\ha
\min_{\substack{\bv \in \BH_0(\ddiv,\oma) \\ \div \bv = 0}}
\|\bv-\grad \pa \times \GRAD\|_{\AAA,\oma}
\\
\nonumber
&\leq
C(\kappa_{\CTa})
\sqrt{\frac{\alpha_{\max,\oma}}{\alpha_{\min,\oma}}}
\ha
\|\grad \pa \times (\grad u-\GRAD)\|_{\AAA,\oma}
\\
\nonumber
&\leq
C(\kappa_{\CTa})
\sqrt{\frac{\alpha_{\max,\oma}}{\alpha_{\min,\oma}}}
\|\grad u-\GRAD\|_{\AAA,\oma},
\end{align}
where we used the fact that $\grad \pa \times \grad u \in \BH_0(\ddiv^0,\oma)$
and \eqref{eq_pa}. For the second term, we first
employ~\cite[Theorem B.1]{chaumontfrelet_vohralik_2023a}, which gives
\begin{equation*}
\ha \|\widetilde \cor_h^{\ba}\|_{\AAA,\oma}
=
C(\kappa_{\CTa})
\sqrt{\frac{\alpha_{\max,\oma}}{\alpha_{\min,\oma}}}
\ha \|\widehat \cor_h\|_{\AAA,\oma}.
\end{equation*}
We then write
\begin{equation*}
\widehat \cor_h = \sum_{\ba \in \CV_h} (\grad \pa \times \grad u-\widehat \cor_h^{\ba}),
\end{equation*}
so that
\begin{equation*}
\ha \|\widetilde \cor_h^{\ba}\|_{\AAA,\oma}
\leq
C(\kappa_{\tCTa})
\frac{\alpha_{\max,\toma}}{\alpha_{\min,\toma}}
\|\grad u-\GRAD\|_{\AAA,\toma},
\end{equation*}
due to to small support of each $\pa$ and~\eqref{tmp_estimate_cor}.
This concludes the proof of~\eqref{eq_flux_eff_vertex}.

For~\eqref{eq_flux_eff_elem}, we simply use that due to~\eqref{eq_pa_pu}
and the definition of $\bphi_h$ in~\eqref{eq_phi}, we have
\begin{equation*}
\|\GRAD-\bphi_h\|_{\AAA,K}
=
\|\sum_{\ba \in \CV(K)}(\pa\GRAD-\bphi_h^{\ba})\|_{\AAA,K}
\leq
\sum_{\ba \in \CV(K)}\|\pa\GRAD-\bphi_h^{\ba}\|_{\AAA,\toma},
\end{equation*}
and the result follows from~\eqref{eq_flux_eff_vertex}.
\end{proof}

\subsection{An alternative error estimator for the IPDG scheme}
\label{section_flux_ipdg}

Here, the curl-free reconstruction $\bphi_h \in \ND_{p+2}(\CT_h) \cap \BH_{\GD}(\ccurl^0,\omega)$
with $\GRAD = \GRAD_h(u_h)$ is used to control
the non-conformity of the solution, i.e. the second term in the
right-hand of the side of the standard Prager--Synge identity
in~\eqref{eq_standard_prager_synge}. For the other term, we follow
the standard flux-equilibration approach
from~\cite{braess_pillwein_schoberl_2009a,destuynder_metivet_1999a,ern_vohralik_2015a}.
We point out that $\GRAD_h(u_h)$ qualifies for the construction
given in Theorem~\ref{theorem_curl_free_reconstruction}, due to~\eqref{eq_stokes_discrete_gradient}.

\begin{remark}[Lower-order lifting]
Given the IPDG solution $u_h \in \CP_p(\CT_h)$ as been computed
(typically through the formulation described in Remark~\ref{remark_standard_IPDG}),
the construction detailed below utilizes the generalized gradient
$\GRAD_h(u_h) \in \BCP_{p-1}(\CT_h)$. However, as already proposed
in~\cite{ern_vohralik_2015a}, the results presented in this section
hold unchanged if the jumps are lifted in a smaller polynomial
space, i.e. if we employ the space $\BCP_0(\CT_h)$ in the definition
of the lifting $\LIFT_h(u_h)$, irrespectively of the polynomial degree
$p$ of the discretization.
\end{remark}

\subsubsection{Flux reconstruction}

For all $\ba \in \CV_h$, the local problem
\begin{equation*}
\bsig_h^{\ba}
\eq
\min_{\substack{
\bw_h \in \RT_{p+1}(\CTa) \cap \BH_0(\ddiv,\oma)
\\
\div \bw_h = \Pi_{hp}(\pa f-\AAA \grad \pa \cdot \GRAD_h(u_h))
)}}
\|\AAA^{-1}\bw_h+\pa\GRAD_h(u_h)\|_{\AAA,\oma}
\end{equation*}
is well-posed since Stokes' compatibility condition is satisfied
due to the Galerkin orthogonality property in~\eqref{eq_ipdg_galerkin_orthogonality}.
Futhermore, whenever $\AAA \in \TCP_0(\CT_h)$, we have
\begin{equation}
\label{eq_flux_equilibration}
\div \bsig_h = \Pi_{hp} f
\end{equation}
for the reconstructed flux
\begin{equation*}
\bsig_h \eq \sum_{\ba \in \CV_h} \bsig_h^{\ba}
\end{equation*}
since the $\pa$ form a partition of unity as per~\eqref{eq_pa_pu}.

\subsubsection{Error estimates}

Once the divergence-conforming and curl-free reconstructions $\bsig_h$ and $\bphi_h$
have been computed, we can easily introduce our error estimator. For all $K \in \CT_h$,
we set
\begin{equation*}
\eta_K^2
\eq
\|\AAA^{-1}\bsig_h+\GRAD_h(u)\|_{\AAA,K}^2
+
\|\GRAD_h(u)-\bphi_h\|_{\AAA,K}^2.
\end{equation*}

\begin{theorem}[Equilibrated estimator for IPDG]
\label{theorem_flux_ipdg}
Assume that $\AAA \in \TCP_0(\CT_h)$. Then, we have
\begin{equation}
\label{eq_flux_ipdg_rel}
\|\GRAD_h(u-u_h)\|_{\AAA,\Omega}^2
\leq
\sqrt{1+\gamma_h^2}
\sum_{K \in \CT_h}
\left \{
\eta_K^2 + \left (\frac{h_K}{\pi}\right )^2\|f-\Pi_{hp} f\|_K^2
\right \}.
\end{equation}
Besides, the following estimate
\begin{equation}
\label{eq_flux_ipdg_eff}
\eta_K
\leq
C(\kappa_{\tomK})
\left (\frac{\alpha_{\max,\tomK}}{\alpha_{\min,\tomK}}\right )^{3/2}
\left \{
\|\GRAD_h(u-u_h)\|_{\AAA,\tomK}^2
+
\left (\frac{h_K}{p}\right )^2 \|f-\Pi_{hp} f\|_{\tomK}^2
\right \}
\end{equation}
holds true for all $K \in \CT_h$.
\end{theorem}

\begin{proof}
The proof of~\eqref{eq_flux_ipdg_rel} simply lies on the standard Prager-Synge identity in~\eqref{eq_standard_prager_synge}.
Indeed, the second term in the right-hand side is bounded by the second term in the estimator as
per Corollary~\ref{corollary_flux_abstract_rel},
whereas the first term is bounded by the first term in the estimator plus
the data oscillation term, which reflects the fact that we don not exactly
have $\div \bsig_h = f$, but only $\div \bsig = \Pi_{hp} f$, as per~\eqref{eq_flux_equilibration}.
The proof that can be found in~\cite[Theorem 3.3]{ern_vohralik_2015a}, or by following the lines
of the proof of Theorem~\ref{theorem_res_mixed_fem}, but using the Poincar\'e
inequality in~\eqref{eq_data_oscillation_h} rather than the on in~\eqref{eq_data_oscillation_hp}.

On the other hand, the lower bounds in~\eqref{eq_flux_ipdg_eff} are obtained by
combining Theorem~\ref{theorem_flux_abstract_eff} for the second term in the
estimator and the state-of-the art theory of equilibrated estimator for conforming methods as
per~\cite{braess_pillwein_schoberl_2009a,
chaumontfrelet_vohralik_2024a,
ern_vohralik_2015a,
ern_vohralik_2021a} for the first term.
\end{proof}

\subsection{An alternative error estimator for mixed FEM discretizations}
\label{section_flux_mixed_fem}

Let $\bphi_h$ be constructed following Theorem~\ref{theorem_curl_free_reconstruction}
with $\GRAD = \AAA^{-1}\bsig_h$. Notice that this is possible whenever $\AAA \in \TCP_0(\CT_h)$
since $\GRAD \in \BCP_p(\CT_h)$ with $(\GRAD,\bxi_h)_\Omega = 0$ for all
$\bxi_h \in \RT_p(\CT_h) \cap \BH_{\GN}(\ddiv^0,\Omega)$ due to~\eqref{eq_orthogonality_mixed_fem}.
Then, we let
\begin{equation*}
\eta_K \eq \|\bsig_h+\AAA\bphi_h\|_{\AAA^{-1},K}
\end{equation*}
for all $K \in \CT_h$, and the following result from Corollay~\ref{corollary_flux_abstract_rel}
and Theorem~\ref{theorem_flux_abstract_eff}, whereby the oscillation term appears exactly as
in the proof of Theorems~\ref{theorem_res_mixed_fem} and~\ref{theorem_flux_ipdg}.

\begin{theorem}[Equilibrated estimator for mixed FEM]
Assume that $\AAA \in \TCP_0(\CT_h)$. Then, we have
\begin{equation*}
\|\AAA\grad u+\bsig_h\|_{\AAA^{-1},\Omega}^2
\leq
\sqrt{1+\gamma_h^2}
\sum_{K \in \CT_h}
\left \{
\eta_K^2 + \left (\frac{h_K}{\pi}\right )^2\|f-\Pi_{hp} f\|_K^2
\right \}.
\end{equation*}
as well as
\begin{equation*}
\eta_K
\leq
C(\kappa_{\tomK})
\left (\frac{\alpha_{\max,\tomK}}{\alpha_{\min,\tomK}}\right )^{3/2}
\left \{
\|\AAA\grad u+\bsig_h\|_{\AAA^{-1},\tomK}^2
+
\left (\frac{h_K}{p}\right )^2 \|f-\Pi_{hp} f\|_{\tomK}^2
\right \}
\end{equation*}
for all $K \in \CT_h$.
\end{theorem}

\subsection{More general coefficients}


For the construction of the estimator above, we need that $\AAA \in \TCP_0(\CT_h)$
if we want to reconstruct $\bsig_h \in \RT_{p+1}(\CT_h)$ for the IPDG scheme and
$\bphi_h \in \ND_{p+1}(\CT_h)$ for the mixed FEM. This assumption can be lifted,
and we can assume that $\AAA \in \TCP_q(\CT_h)$ if we are ready to using higher-order
spaces for the reconstructions in such a way that $\bsig_h \in \RT_{p+q+1}(\CT_h)$
for the IPDG scheme and $\bphi_h \in \ND_{p+q+1}(\CT_h)$ for the mixed FEM.

\section{Alternative equilibration techniques}
\label{section_alternative}

In this section, we explore an alternative to the flux equilibration
technique described above. This alternative is less expensive and
easier to implement. It preserves the polynomial-degree-robustness, but
introduces an a priori unknown constant in the upper bound.

\subsection{Tools}

We start by listing some key technical results useful throughout this section.

\subsubsection{Edge patches}

For $\ell \in \CE_h$ $\CTe$ is the edge patch collecting elements
$K \in \CT_h$ such that $\ell \in \CE(K)$, and we denote by $\ome \subset \Omega$ the
associated open domain. For an element $K \in \CT_h$ we then also employ the notation
$\omKE = \cup_{\ell \in \CE(K)} \ome$.

\subsubsection{Edge partition of unity}

For each edge $\ell \in \CE_h$, we denote by $\pe$ the ``edge'' function
associated with $\ell$, i.e. the only element of $\ND_1(\CT_h) \cap \BH(\ccurl,\Omega)$
such that $(\pe,1)_\ell = |\ell|\delta_{\ell,\ell'}$ for all $\ell' \in \CE_h$,
where $|\ell|$ is the length of the edge. The suppot of $\pe$ coincides with
$\ome$, and we have
\begin{equation}
\label{eq_bounds_pe}
\|\pe\|_{\BL^\infty(\Omega)}
\leq
C(\kappa_{\ome}),
\qquad
\|\curl\pe\|_{\BL^\infty(\Omega)}
\leq
C(\kappa_{\ome})
\he^{-1}.
\end{equation}
In addition, if $\bw \in \BL^2(\Omega)$, then
\begin{equation}
\label{eq_pe_pu}
\bw = \sum_{\ell \in \CE_h} (\bw \cdot \te) \pe.
\end{equation}

\subsubsection{Local function spaces and inequalities}

For each $\ell \in \CE_h$, if $\ell \not \subset \overline{\GN}$, we let
$\gamma_\ell \eq \partial \ome$ and $\gamma_\ell^{\rm c} \eq \emptyset$.
On the other hand, if $\ell \subset \overline{\GN}$, then we let
$\gamma_\ell \eq \partial \ome \setminus \gamma_\ell^{\rm c}$, where
$\gamma_\ell^{\rm c}$ corresponds to the faces $F \in \CF_h$ having $\ell$ as
an edge and lying in $\overline{\GN}$. Notice that except when the mesh is very
coarse, $\gamma_\ell^{\rm c} = \partial \ome \cap \GN$. We then set
$\BH_0(\ddiv,\ome) \eq \BH_{\gamma_\ell}(\ddiv,\ome)$,
$H^1_\star(\ome) \eq H^1_{\gamma_\ell^{\rm c}}(\ome)$,
and $L^2_0(\ome) \eq \div \BH_0(\ddiv,\ome)$.
$L^2_0(\ome) = L^2(\ome)$ when $\gamma_\ell^{\rm c} \neq \emptyset$,
and it contains functions with vanishing mean-value otherwise.

We observe that whenever $r \in H^1_\star(\ome)$, we have
$r\pe \subset \BH_{\gaeD}(\ccurl,\ome)$,
where $\gaeD$ collects the faces of $F \in \CF_h$ having $\ell$
as an edge and lying in $\overline{\GD}$. Besides, the estimate
\begin{equation}
\label{eq_conte}
\|\curl (r\pe)\|_{\ome} \leq C(\kappa_{\CTe}) \|\grad r\|_{\ome}
\qquad
\forall r \in H^1_\star(\ome) \cap L^2_0(\ome)
\end{equation}
may be obtained by combining the product rule, the bounds in~\eqref{eq_bounds_pe}
and a Poincar\'e inequality in the patch $\ome$. The resulting constant only depends
on the shape-regularity of the mesh, see~\cite[Theorem 3.2]{chaumontfrelet_2023b}.
We emphasize that a Poincar\'e inequality is always available in $H^1_\star(\ome) \cap L^2_0(\ome)$
since either $\gamma_\ell^{\rm c} \neq \emptyset$ and the functions have vanishing traces in
a subset of $\partial \ome$, or they have vanishing mean-value.

\subsubsection{Regular decomposition}
\label{section_regular_decomposition_flux}

Similar to the discussion in Section~\ref{section_decomposition_schoberl}
for the residual-based estimator, we are going to use a regular decomposition
of a vector field in $\BH_{\GN}(\ccurl,\Omega)$ with suitable orthogonality cosntraint
into a $\BH^1_{\GN}(\Omega)$ component and a gradient. Here, however, the resulting
stability constant is the only unknown quantity in the reliability bound, so that
we take great care to identify it sharply. We refer the reader
to~\cite[Remark 5.2]{chaumontfrelet_2023b} for more details.

We introduce
\begin{equation*}
\CLA
\eq
\sup_{\substack{\bg \in \BX_{hq}^\perp(\Omega)}}
\sup_{\substack{\bphi \in \BH_{\GN}(\ccurl,\Omega) \\ \|\curl \bphi\|_{\AAA^{-1},\Omega} = 1}}
\min_{\substack{\bthe \in \BH_{\GN}^1(\Omega) \\ (\bg,\curl \bthe)_\Omega = (\bg,\curl\bphi)_\Omega}}
\|\tgrad \bthe\|_{\AAA^{-1},\Omega}
\end{equation*}
where, for $q \geq 1$,
\begin{equation*}
\BX_{hq}^\perp(\Omega)
\eq
\{
\bg \in \BL^2(\Omega)
\; | \;
(\bg,\bxi_h)_\Omega = 0
\;
\forall \bxi_h \in \RT_q(\CT_h) \cap \BH_{\GN}(\ddiv,\Omega)
\}.
\end{equation*}
As we show below, this constant is (at least) controlled uniformly
by the shape-regularity parameter and the contrast in the diffusion
coefficient.

\begin{lemma}[Regular decomposition constant]
We have
\begin{equation*}
\CLA
\leq
C(\kappa_{\CT_h}) \max_{K \in \CT_h} \sqrt{\frac{\alpha_{\max,\tomK}}{\alpha_{\min,\tomK}}}.
\end{equation*}
\end{lemma}

\begin{proof}
Let $\bg \in \BX_{hq}^\perp(\Omega)$ and $\bphi \in \BH_{\GN}(\ccurl,\Omega)$
with $\|\curl \bphi\|_{\AAA^{-1},\Omega} = 1$. According to~\cite[Theorem 1]{schoberl_2008a},
we can find $\bphi_h \in \ND_1(\CT_h) \cap \BH_{\GN}(\ccurl,\Omega)$ and
$\bthe \in \BH^1_{\GN}(\Omega)$ such that
\begin{equation*}
\curl \bthe = \curl(\bphi-\bphi_h),
\end{equation*}
and
\begin{equation*}
\|\tgrad \bthe\|_K
\leq
C(\kappa_{\CT_h}) \|\curl \bphi\|_{\tomK}.
\end{equation*}
We now observe that
\begin{equation*}
(\bg,\curl \bthe)_\Omega = (\bg,\curl(\bphi-\bphi_h))_\Omega = (\bg,\curl\bphi)_\Omega,
\end{equation*}
since $\curl \bphi_h \in \RT_1(\CT_h) \cap \BH_{\GN}(\ddiv,\Omega)$. Hence,
$\bthe$ is in the minimization set, and it remains to estimate
$\|\tgrad \bthe\|_{\AAA^{-1},\Omega}$. We have
\begin{multline*}
\|\tgrad \bthe\|_{\AAA^{-1},\Omega}^2
\leq
\sum_{K \in \CT_h}
\|\tgrad \bthe\|_{\AAA^{-1},K}^2
\leq
\sum_{K \in \CT_h}
\alpha_{\min,\tomK}^{-1}
\|\tgrad \bthe\|_{K}^2
\\
\leq
C(\kappa_{\CT_h})
\sum_{K \in \CT_h}
\alpha_{\min,\tomK}^{-1}
\|\curl \bphi\|_{\tomK}^2
\leq
C(\kappa_{\CT_h})
\sum_{K \in \CT_h}
\frac{\alpha_{\max,\tomK}}{\alpha_{\min,\tomK}}
\|\curl \bphi\|_{\AAA^{-1},\tomK}^2
\\
\leq
C(\kappa_{\CT_h})
\max_{K \in \CT_h}
\frac{\alpha_{\max,\tomK}}{\alpha_{\min,\tomK}}
\|\curl \bphi\|_{\AAA^{-1},\Omega}^2,
\end{multline*}
where we used the fact that the number of overlaps is controlled by
the shape-regularity parameter $\kappa_{\CT_h}$.
\end{proof}

\subsection{Construction of the estimator}

We are now ready to propose the last error estimator studied in this
work. To motivate the construction, we start by establishing an alternative
equilibration principle.

\begin{theorem}[Alternative equilibration]
\label{theorem_aflux_equilibration}
Let $\GRAD \in \BL^2(\Omega)$ such that $(\GRAD,\bxi_h)_\Omega = 0$
for all $\bxi_h \in \RT_1(\CT_h) \cap \BH_{\GN}(\ddiv^0,\Omega)$. Then
\begin{equation*}
\min_{s \in H^1_{\GD}(\Omega)} \|\GRAD-\grad s\|_{\AAA,\Omega}^2
\leq
\CLA^2
\sum_{k=1}^3
\min_{\bsig^k \in \BH_{\GD}(\ddiv^0,\Omega)}
\|\GRAD \times \cank - \bsig^k\|_{\AAA,\Omega}^2.
\end{equation*}
\end{theorem}

\begin{proof}
Let $\bthe \in \BH^1_{\GN}(\Omega)$. We have
\begin{equation*}
(\GRAD,\curl \bthe)_\Omega
=
\sum_{k=1}^3
(\GRAD,\curl (\bthe_k \cank))_\Omega
=
\sum_{k=1}^3
(\GRAD,\grad \bthe_k \times \cank)_\Omega
=
-
\sum_{k=1}^3
(\GRAD\times \cank,\grad \bthe_k)_\Omega.
\end{equation*}
If $\bsig^k \in \BH_{\GD}(\ddiv^0,\Omega)$, then
\begin{equation*}
(\GRAD\times \cank,\grad \bthe_k)_\Omega.
=
(\GRAD\times \cank-\bsig^k,\grad \bthe_k)_\Omega,
\end{equation*}
so that
\begin{align*}
|(\GRAD\times \cank,\grad \bthe_k)_\Omega|
&\leq
\min_{\bsig^k \in \BH_{\GD}(\ddiv^0,\Omega)}
\|\AAA(\GRAD \times \cank - \bsig^k)\|_{\AAA^{-1},\Omega}
\|\grad \bthe_k\|_{\AAA^{-1},\Omega}
\\
&=
\min_{\bsig^k \in \BH_{\GD}(\ddiv^0,\Omega)}
\|\GRAD \times \cank - \bsig^k\|_{\AAA,\Omega}
\|\grad \bthe_k\|_{\AAA^{-1},\Omega}.
\end{align*}
Therefore,
\begin{equation*}
(\GRAD,\curl \bthe)_\Omega
\leq
\left (
\sum_{k=1}^3
\min_{\bsig^k \in \BH_{\GD}(\ddiv^0,\Omega)}
\|\GRAD \times \cank - \bsig^k\|_{\AAA,\Omega}^2
\right )^{1/2}
\left (
\sum_{k=1}^3
\|\grad \bthe_k\|_{\AAA^{-1},\Omega}^2
\right )^{1/2}.
\end{equation*}
The left-hand side is now controlled for all~$\bthe \in \BH^1_{\GN}(\Omega)$.
To apply~\eqref{eq_generalized_prager_synge}, we need to control it for all
$\BH_{\GN}(\ccurl,\Omega)$ test functions, which can be done at the price
of the constant $\CLA$ following the discution in
Section~\ref{section_regular_decomposition_flux}.
\end{proof}

\begin{subequations}
\label{eq_definition_sigh}
Following, Theorem~\ref{theorem_aflux_equilibration}, we propose the following construction.
Assuming that $\GRAD \in \BCP_q(\CT_h)$ for some $q \geq 0$, for each $\ell \in \CE_h$,
we introduce
\begin{equation}
\label{eq_definition_sighe}
\bsig_{h,\ell}
\eq
\arg \min_{\substack{
\bv_h \in \BH_0(\ddiv,\ome) \cap \RT_{q+2}(\CTe)
\\
\div \bv_h = \GRAD \cdot \curl \pe
}}
\|\GRAD \times \pe - \bv_h\|_{\AAA,\ome}.
\end{equation}
Then, for $1 \leq k \leq 3$, we combined these local contributions into global objects
\begin{equation}
\label{eq_definintion_sighk}
\bsig_h^k
\eq
\sum_{\ell \in \CE_h} (\te \cdot \cank) \bsig_{h,\ell}.
\end{equation}
\end{subequations}

\begin{theorem}[Construction of the estimator]
\label{theorem_aflux_estimator}
Assume that $\GRAD \in \BCP_q(\CT_h)$ for some $q \geq 0$ with
$(\GRAD,\bxi_h)_\Omega = 0$ for all $\bxi_h \in \RT_1(\CT_h) \cap \BH_{\GN}(\ddiv^0,\Omega)$.
For each $\ell \in \CE_h$, the local problem in \eqref{eq_definition_sighe} is well-posed
and uniquely defines $\bsig_{h,\ell}$. Besides, for $1 \leq k \leq 3$, we have
$\bsig_h^k \in \BH_{\GD}(\ddiv^0,\Omega)$.
\end{theorem}

\begin{proof}
As far as the well-posedness of the local problems in~\eqref{eq_definition_sighe}
is concerned, we simply need to check that
\begin{equation*}
(\GRAD \cdot \curl \pe,1)_{\ome} = (\GRAD,\curl \pe)_{\Omega} = 0
\end{equation*}
whenever $\ell \not \subset \overline{\GN}$.
This, however, immediately follows by assumption, since
$\curl \pe \in \RT_1(\CT_h) \cap \BH_{\GN}(\ddiv^0\Omega)$.

Let $1 \leq k \leq 3$.
The fact that $\bsig_h^k \in \BH_{\GD}(\ddiv,\Omega)$ readily follows from the fact
that each $\bsig_{h,\ell} \in \BH_0(\ddiv,\ome)$. Besides, the constraint $\div \bsig_{h,\ell} = 0$
may be obtained directly from the partition of unity property of the edge functions
stated in~\eqref{eq_pe_pu}. Indeed, we have
\begin{equation*}
\div \bsig_h^k
=
\GRAD \cdot \curl \sum_{\ell \in \CE_h} (\te \cdot \cank) \pe
=
\GRAD \cdot \curl \cank
=
0.
\end{equation*}
\end{proof}

\begin{remark}[Reconstruction in smaller discrete spaces]
It is also possible to reconstruct $\bsig_{h,\ell}$ in $\BCP_{q+1}(\CTe)$
rather than $\RT_{q+2}(\CTe)$, resulting in a more cost effective reconstruction.
However, while it is still possible to show that the resulting estimator is
efficient, its polynomial-degree-robustness cannot be directly inferred from
the literature in this case.
\end{remark}

\subsection{Reliability}

As a direct consequence of Theorems~\ref{theorem_aflux_equilibration}
and~\ref{theorem_aflux_estimator}, we have the following reliability result.

\begin{theorem}[Reliability]
Let $\GRAD \in \BCP_q(\CT_h)$ for some $q \geq 0$ such that $(\GRAD,\bxi_h)_\Omega = 0$
for all $\bxi_h \in \RT_1(\CT_h) \cap \BH_{\GN}(\ddiv^0,\Omega)$,
and assume that for $1 \leq k \leq 3$, $\bsig_h^k \in \RT_{q+1}(\CT_h)$
has been constructed following~\eqref{eq_definition_sigh}. Then, we have
\begin{equation*}
\min_{s \in H^1_{\GD}(\Omega)} \|\GRAD-\grad s\|_{\AAA,\Omega}^2
\leq
\CLA^2
\sum_{k=1}^3 \|\GRAD \times \cank - \bsig_h^k\|_{\AAA,\Omega}^2.
\end{equation*}
\end{theorem}

\subsection{Efficiency}

We now establish an efficiency result. We start with a preliminary
result concerning the local contributions $\bsig_{h,\ell}$.

\begin{lemma}[Efficiency of contributions]
Assume that $\GRAD \in \BCP_q(\CT_h)$ for some $q \geq 0$.
Then for all $\ell \in \CE_h$, the local lower bound
\begin{equation}
\label{eq_aflux_efficiency_contribution}
\|\GRAD \times \pe-\bsig_{h,\ell}\|_{\AAA,\ome}
\leq
C(\kappa_{\CTe}) \sqrt{\frac{\alpha_{\max,\ome}}{\alpha_{\min,\ome}}}
\min_{s \in H^1_{\gaeD}(\ome)} \|\GRAD-\grad s\|_{\AAA,\ome}
\end{equation}
holds true.
\end{lemma}

\begin{proof}
We start by considering the ``continuous'' version of the minimization problem
in~\eqref{eq_definition_sighe}, namely:
\begin{equation*}
\bsig_{\ell}
\eq
\arg \min_{\substack{
\bv \in \BH_0(\ddiv,\ome)
\\
\div \bv = \GRAD \cdot \curl \pe
}}
\|\GRAD \times \pe - \bv\|_{\AAA,\ome}.
\end{equation*}

Following~\cite{braess_pillwein_schoberl_2009a,%
chaumontfrelet_2023b,ern_vohralik_2021a,chaumontfrelet_vohralik_2024a},
a discrete stable minimization result shows that
\begin{equation}
\label{tmp_aflux_eff1}
\|\GRAD \times \pe-\bsig_{h,\ell}\|_{\AAA,\ome}
\leq
C(\kappa_{\CTe})
\sqrt{\frac{\alpha_{\max,\ome}}{\alpha_{\min,\ome}}}
\|\GRAD \times \pe-\bsig_{\ell}\|_{\AAA,\ome},
\end{equation}
and it remains to estimate the right-hand side.

The associated Euler-Lagrange equations are to find
$\bsig_\ell \in \BH_0(\ddiv,\ome)$ and $q_\ell \in L^2_0(\ome)$
such that
\begin{equation}
\label{tmp_euler_lagrange_sigl}
\left \{
\begin{array}{rcl}
(\AAA \bsig_\ell,\bv)_{\ome}+(q_\ell,\div \bv)_{\ome} &=& (\AAA\GRAD \times \pe,\bv)_{\ome}
\\
(\div \bsig_\ell,r)_{\ome} &=& (\GRAD \cdot \curl \pe,r)_{\ome}
\end{array}
\right .
\end{equation}
for all $\bv \in \BH_0(\ddiv,\ome)$ and $r \in L^2_0(\ome)$.

From the first equation in~\eqref{tmp_euler_lagrange_sigl}, we read that
$q_\ell \in H^1_\star(\ome)$ with $\AAA^{-1} \grad q_\ell = \bsig_\ell-\GRAD \times \pe$.
In particular, we have
\begin{equation}
\label{tmp_aflux_eff2}
\|\GRAD \times \pe-\bsig_{\ell}\|_{\AAA,\ome}
=
\|\grad q_\ell\|_{\AAA^{-1},\Omega}.
\end{equation}
On the other hand, taking arguments in $r \in H^1_\star(\ome) \cap L^2_0(\ome)$
in the second equation shows that
\begin{equation*}
-(\bsig_\ell,\grad r)_{\ome} = (\GRAD \cdot \curl \pe,r)_{\ome}
\end{equation*}
and it follows that
\begin{equation*}
-(\AAA^{-1}\grad q_\ell+\GRAD \times \pe,\grad r)_{\ome} = (\GRAD \cdot \curl \pe,r)_{\ome},
\end{equation*}
which we rewrite as
\begin{multline*}
(\AAA^{-1}\grad q_\ell,\grad r)_{\ome}
=
-(\GRAD \cdot \curl \pe,r)_{\ome}-(\GRAD \times \pe,\grad r)_{\ome}
\\
=
-(\GRAD,r \curl \pe+\grad r \times \pe)_{\ome}
=
-(\GRAD,\curl(r\pe))_{\ome}
=
-(\GRAD-\grad s,\curl(r\pe))_{\ome}
\end{multline*}
for all $s \in H^1_{\gaeD}(\ome)$.
Recalling that $q_\ell \in H^1_\star(\ome) \cap L^2_0(\ome)$ and taking $r = q_\ell$,
it follows that
\begin{equation}
\label{tmp_aflux_eff3}
\|\grad q_\ell\|_{\AAA^{-1},\ome}
\leq
C(\kappa_{\CTe}) \min_{s \in H^1_{\gaeD}(\ome)} \|\GRAD-\grad s\|_{\AAA,\ome},
\end{equation}
where we employed~\eqref{eq_conte}. The estimate in~\eqref{eq_aflux_efficiency_contribution}
now follows from~\eqref{tmp_aflux_eff1},~\eqref{tmp_aflux_eff2} and~\eqref{tmp_aflux_eff3}.
\end{proof}

We are now ready to give the main efficiency estimate of this section.

\begin{theorem}[Efficiency]
Let $\GRAD \in \BCP_q(\CT_h)$ for some $q \geq 0$.
For all $K \in \CT_h$, the local lower bound
\begin{equation}
\label{eq_aflux_efficiency}
\sum_{k=1}^3\|\GRAD \times \cank-\bsig_h^k\|_{\AAA,K}^2
\leq
C(\kappa_{\CTe})
\frac{\alpha_{\max,\omK}}{\alpha_{\min,\omK}} \|\GRAD-\grad u\|_{\AAA,\omKE}^2
\end{equation}
holds true.
\end{theorem}

\begin{proof}
Fix $K \in \CT_h$ and $1 \leq k \leq 3$. We use the partition of unity property~\eqref{eq_pe_pu}
of the edge functions to write that
\begin{equation*}
\GRAD \times \cank|_K = \sum_{\ell \in \CE(K)} (\te \cdot \cank) \GRAD \times \te,
\end{equation*}
and it follows from the definition of $\bsig_h^k$ that
\begin{multline*}
\|\GRAD \times \cank-\bsig_h^k\|_{\AAA,K}
=
\|\sum_{\ell \in \CE(K)} (\te \cdot \cank)(\GRAD \times \pe -\bsig_{h,\ell})\|_{\AAA,K}
\leq
C(\kappa_{\CTe})
\sum_{\ell \in \CE(K)}
\left \|\GRAD \times \pe -\bsig_{h,\ell}\right \|_{\AAA,\ome},
\end{multline*}
and~\eqref{eq_aflux_efficiency} follows from~\eqref{eq_aflux_efficiency_contribution},
with $s=u$.
\end{proof}

\subsection{Alternative error estimators for IPDG and mixed FEM}

The application of the alternatively equilibrated fluxes to the IPDG
scheme and the mixed FEM discretization follows closely the procedure
described in Sections~\ref{section_flux_ipdg} and~\ref{section_flux_mixed_fem}
and is not detailed here for shortness.

\bibliographystyle{amsplain}
\bibliography{files/bibliography}

\appendix

\section{Control of $\gamma_h$}
\label{appendix_gamma}

\begin{proof}[Proof of Lemma~\ref{lemma_gamma}]
There exists a integer $N \geq 0$ and, for $1 \leq j \leq N$, $N$ oriented surfaces
$\Sigma^j$ with unit normal vector $\bn^j$ such that
the following two statements are equivalent: (a)
$\bxi \in \BH_{\GN}(\ddiv^0,\Omega)$
is such that
there exists $\bthe \in \BH_{\GN}(\ccurl,\Omega)$ with $\curl \bthe = \bxi$
and (b) the periods
\begin{equation*}
\int_{\Sigma^j} \bxi \cdot \bn^j
\end{equation*}
vanish for all $j \in \{1,\dots,N\}$, see e.g.~\cite{duff_1952a} and the discussion
in~\cite[Section 2.3]{chaumontfrelet_2025a}.

In fact, we can assume without loss of generality that these surfaces are
exactly covered by mesh faces. It is then easy to see that for $1 \leq j \leq N$
there exists a lowest-order Raviart--Thomas function
$\bsig_h^j \in \RT_1(\CT_h) \cap \BH_{\GN}(\ddiv^0,\Omega)$ such that
\begin{equation}
\label{eq_definition_bsighj}
\int_{\Sigma^k} \bsig_h^j \cdot \bn^k = \delta_{jk}
\end{equation}
for $1 \leq k \leq N$.

Let us now consider an arbitrary element $\lba \in \LBA$. Following the above discussion,
we can find $\lba_1 \in \RT_1(\CT_h) \cap \BH_{\GN}(\ddiv^0,\Omega)$ such that
\begin{equation*}
\int_{\Sigma^k} \lba_1 \cdot \bn^k = \int_{\Sigma^k} \lba \cdot \bn^k
\end{equation*}
for $1 \leq k \leq N$, and it follows that
$\lba-\lba_1 = \curl \bthe$ for some $\bthe \in \BH_{\GN}(\ccurl,\Omega)$.
We now invoke the commuting interpolator
$J_h^{\rm c}: \BH_{\GN}(\ccurl,\Omega) \to \ND_1(\CT_h) \cap \BH_{\GN}(\ccurl,\Omega)$
from~\cite{chaumontfrelet_vohralik_2024_proj_hcurl}, and we let
\begin{equation*}
\lba_h = \lba_1+\curl J_h^{\rm c}(\bthe).
\end{equation*}
We first note that
\begin{equation*}
\lba-\lba_h = \curl (\bthe-J_h^{\rm c}(\bthe)),
\end{equation*}
so that in particular $\lba-\lba_h \in \curl \BH_{\GN}(\ccurl,\Omega)$.
In then follows from the commutation property of $J_h^{\rm c}$
with its sibling operator in $\BH_{\GN}(\ddiv,\Omega)$
and~\cite[Theorem 3.2, eq. (3.7)]{ern_gudi_smears_vohralik_2022a} that
\begin{equation*}
\|\lba-\lba_h\|_{\AAA^{-1},\Omega}
\leq
C(\kappa_{\CT_h}) \sqrt{\max_{K \in \CT_h}\frac{\alpha_{\max,\omK}}{\alpha_{\min,\omK}}}
\min_{\bv_h \in \RT_1(\CT_h) \cap \BH_{\GN}(\ddiv,\Omega)}
\|\lba-\bv_h\|_{\AAA^{-1},\Omega}.
\end{equation*}
At that point,~\eqref{eq_gamma_st} follows by picking $\bv_h = \bzero$
in the minimization set. To establish~\eqref{eq_gamma_cv}, we simply
need to observe that $\LBA \subset \BH^s(\Omega)$ for some $s > 0$.
Indeed, the elements of $\LBA$ are of the form $\BA^{-1}\bxi$,
where $\BA\bxi \in \BH_{\GN}(\ddiv^0,\Omega)$, and $\bxi \in \BH_{\GD}(\ccurl^0,\Omega)$.
Then, such embedding is well-known~\cite{bonito_germond_luddens_2013a}.
\end{proof}

\end{document}